\documentclass[10pt]{amsart}
\usepackage[latin1]{inputenc}
\usepackage{fig4tex184}
\usepackage{epsfig}
\usepackage{color}
\usepackage{amsthm}
\usepackage{amsmath}
\usepackage{amsfonts}
\usepackage{amssymb}
\usepackage{graphicx}
\makeatletter
\@addtoreset{equation}{section} 
\makeatother
\newtheorem{lem}{Lemma}[section]
\newtheorem{prop}{Proposition}[section]
\newtheorem{cor}{Corollary}[section]
\newtheorem{thm}{Theorem}[section]
\newtheorem{rem}{Remark}[section]
\newfont{\sBlackboard}{msbm10 scaled 900}

\newcommand{\ud}     {{\rm d}}
\newcommand{\mylabel}[1]{\label{#1}
            \ifx\undefined\stillediting
            \else \fbox{$#1$}\fi }

\newcommand{\EEQ}{\end{equation}}
\newcommand{\rfb}[1]{\mbox{\rm
   (\ref{#1})}\ifx\undefined\stillediting\else:\fbox{$#1$}\fi}

\newfont{\Blackboard}{msbm10 scaled 1200}

\newfont{\roma}{cmr10 scaled 1200}

\def\CC{\rm \hbox{C\kern-.56em\raise.4ex
         \hbox{$\scriptscriptstyle |$}\kern+0.5 em }}
\newcommand{\bt}{\begin{Theorem}}
\newcommand{\et}{\end{Theorem}}
\newcommand{\br}{\begin{remark}}
\newcommand{\er}{\end{remark}}
\newcommand{\bc}{\begin{Corollary}}
\newcommand{\ec}{\end{Corollary}}
\newcommand{\el}{\end{Lemma}}
\newcommand{\bd}{\begin{definition}}
\newcommand{\ed}{\end{definition}}

\newcommand{\N}  {\mathbb{N}}
\newcommand{\R}  {\mathbb{R}}

\newcommand{\mm}    {{\hbox{\hskip 0.5pt}}}
\newcommand{\m}     {{\hbox{\hskip 1pt}}}

\newcommand{\bluff} {{\hbox{\raise 15pt \hbox{\mm}}}}

\newcommand{\rarrow} {{\,\rightarrow\,}}

\makeatletter
\def\section{\@startsection {section}{1}{\z@}{-3.5ex plus -1ex minus
    -.2ex}{2.3ex plus .2ex}{\large\bf}}
\makeatother
\def\ds{\displaystyle}
\newcommand{\re}{\mathrm{Re}}
\newcommand{\im}{\mathrm{Im}}
\newcommand{\e}{\mathrm{e}}
\begin{document}

\thispagestyle{empty}
\title[]{Fractional-feedback stabilization for a class of evolution systems}
\author{Ka\"{\i}s AMMARI}
\address{UR Analysis and Control of PDEs, UR 13ES64, Department of Mathematics, Faculty of Sciences of Monastir, University of Monastir, 5019 Monastir, Tunisia and LMV/UVSQ/Paris-Saclay, France} \email{kais.ammari@fsm.rnu.tn}

 \author{Hassine Fathi}
\address{UR Analysis and Control of PDEs, UR 13ES64, Department of Mathematics, Faculty of Sciences of Monastir, University of Monastir, 5019 Monastir, Tunisia} \email{fathi.hassine@fsm.rnu.tn}

\author{Luc ROBBIANO}
\address{Laboratoire de Math\'ematiques, Universit\'e de Versailles Saint-Quentin en Yvelines, 78035 Versailles, France}
\email{luc.robbiano@uvsq.fr}

\date \today

\begin{abstract}
We study the problem of stabilization for a class of evolution systems with fractional-damping. After writing the equations as an augmented system we prove in this article first that the problem is well posed. Second, using the LaSalle's invariance principle we show that the system is strongly stable. Then, based on a resolvent approach we show a luck of uniform stabilization. Next, using multiplier techniques combined with the frequency domain method, we shall give a polynomial stabilization result under some consideration on the stabilization of an auxiliary dissipating system. Finally, we give some applications to the wave equation. 
\end{abstract}


\maketitle

\tableofcontents
\section{Introduction}
\setcounter{equation}{0}

In recent years, fractional calculus has been increasingly applied in different fields of science \cite{magin,tarasov,VMK}. Physical phenomena related to electromagnetism, propagation of energy in dissipative systems, thermal stresses, models of porous electrodes, relaxation vibrations, viscoelasticity and thermoelasticity are successfully described by fractional differential equations \cite{GB,ML,matignon1,matignon2}. Fractional calculus allows for the investigation of the nonlocal response of mechanical systems, this is the main advantage when compared to the classical calculus.

In the literature, a number of definitions of the fractional derivatives have been introduced, namely the Hadamard, Erdelyi-Kober, Riemann-Liouville, Riesz, Weyl, Gr\"unwald-Letnikov, Jumarie and the Caputo representation. A thorough analysis of fractional dynamical systems is necessary to achieve an appropriate definition of the fractional derivative. For example, the Riemann-Liouville definition entails physically unacceptable initial conditions (fractional order initial conditions); conversely, for the Caputo representation, the initial conditions are expressed in terms of integer-order derivatives having direct physical significance; this definition is mainly used to include memory effects. Recently, Michele Caputo and Mauro Fabrizio in \cite{CF} presented a new definition of the fractional derivative without a singular kernel; this derivative possesses very interesting properties, for instance the possibility to describe fluctuations and structures with different scales. Furthermore, this definition allows for the description of mechanical properties related to damage, fatigue and material heterogeneities.

Let $H$ be a Hilbert space equipped with the norm $\|\,.\,\|_H$, and let $A:\mathcal{D}(A)\subset H \rightarrow H$ be a self-adjoint and strictly positive operator on $H$. We introduce the scale of Hilbert spaces $H_{\beta}$, $\beta\in\R$, as follows: for every $\beta \geq 0$, $H_{\beta}={\mathcal D}(A^{\beta})$, with the norm $\|z\|_{\beta}=\|A^{\beta} z\|_{H}$. The space $H_{-\beta}$ is defined by duality with respect to the pivot space $H$ as follows: $H_{-\beta} =H_{\beta}^*$ for $\beta>0$. The operator $A$ can be extended (or restricted) to each $H_\beta$, such that it becomes a bounded operator
$$
A:H_\beta\rarrow H_{\beta-1},\quad \forall\,\beta\in\R \m.
$$

Let a bounded linear operator $B:U\rarrow H_{-\frac{1}{2}}$, where $U$ is another Hilbert space which will be identified with its dual.

The system we consider here is described by:
\begin{equation}\label{IFF1}
\left\{
\begin{array}{ll}
\partial_{t}^{2}u(t)+Au(t)+BB^{*}\partial^{\alpha,\eta}_{t}u(t)=0,\; t > 0, 
\\
u(0)=u^{0},\;\partial_{t}u(0)=u^{1},
\end{array}
\right.
\end{equation}
where $\partial_{t}^{\alpha,\eta}$ denoted the fractional derivative defined by
\begin{equation}\label{IFF9}
\partial_{t}^{\alpha,\eta}v(t)=\frac{1}{\Gamma (1-\alpha)}\,\int_{0}^{t} (t-s)^{-\alpha}\,e^{-\eta(t-s)}\,v^{\prime}(s) \,\ud s,\;0<\alpha<1,\;\eta\geq 0.
\end{equation}
We define also the following  exponentially modified fractional integro-differential operators
\begin{equation*}
I^{\alpha,\eta}v(t)=\frac{1}{\Gamma (\alpha)}\,\int_{0}^{t} (t-s)^{\alpha-1}\,e^{-\eta(t-s)}\,v(s)\,\ud s,\;0<\alpha<1,\;\eta\geq 0.
\end{equation*}
With these notations we have
\begin{equation}\label{IFF10}
\partial_{t}^{\alpha,\eta}v (t)=I^{1-\alpha,\eta}v^{\prime}(t).
\end{equation} 
There are many definitions for fractional derivatives \cite{das}, among which Riemann-Liouville definition and Caputo definitions are most widely used \cite{MJBSR}. The latter has the same Laplace transform as the integer order one, so it is widely used in control theory. In this paper, the fractional derivative damping force is regarded as a control force to study the properties of free damped vibration of the system, so the Caputo definition is used here.

Noting that the case of the wave equation with boundary fractional damping have treated in \cite{mbodje,mbodjemontsey} where it is proven the strong stability and the lack of uniform stabilization. However, the case of the plate equation or the beam equation with boundary fractional damping was treated in \cite{AAB} where in addition of that using the domain frequency method it was shown that the energy is  polynomially stable.

The main result of this paper concerns the precise asymptotic behavior of the solutions of \eqref{IFF2}-\eqref{IFF4}. 
Our technique is based on a resolvent estimate.   

\medskip

This paper is organized as follows. In section \ref{ASFF} we reformulate problem \eqref{IFF1} into an augmented system. In Section \ref{WFF}, we give the proper functional setting for the augmented model \eqref{IFF2}-\eqref{IFF4}, and prove that this system is well-posed. In Section \ref{SSFF}, we establish a resolvent estimate which correspond to the system \eqref{IFF2}-\eqref{IFF4} and by resolvent method we give the explicit decay rate of the energy of the solutions of \eqref{IFF2}-\eqref{IFF4}. At the end we give some applications to the wave equation.
\section{Augmented model}\label{ASFF}
In this section we reformulate \eqref{IFF1} into an augmented system. our main result is the following. 
\begin{prop}\label{ASFF3}
We set the constant
\begin{equation*}
\gamma=\frac{2\sin(\alpha\pi)\,\Gamma(\frac{1}{2}+1)}{\pi^{\frac{1}{2}+1}},
\end{equation*}
and
 we define the function
\begin{equation*}
p(\xi)=|\xi|^{\frac{2 \alpha-1}{2}}.
\end{equation*}
Then the relation between the input $U$ and the output $O$ of the following system
\begin{equation}\label{ASFF1}
\left\{\begin{array}{ll}
\partial_{t}\varphi(t,\xi)+(|\xi|^2+\eta)\varphi(t,\xi)=p(\xi)U(t)&\forall\,\xi\in\R,\;t>0
\\
\varphi(0,\xi)=0&\forall\,\xi\in\R
\\
\ds O(t)=\gamma\int_{\R}p(\xi)\varphi(t,\xi)\,\ud\xi,&\forall\,t\geq 0,
\end{array}\right.
\end{equation}
where $U\in\mathcal{C}^{0}([0,+\infty))$, is given by
\begin{equation}\label{ASFF2}
O(t)=I^{1-\alpha,\eta}U(t).
\end{equation}
\end{prop}
\begin{proof}
Solving equation \eqref{ASFF1}, we obtain
$$
\varphi(t,\xi)=p(\xi)\int_{0}^{t}\e^{-(|\xi|^{2}+\eta)(t-s)}U(s)\,\ud s.
$$
If follows from the third line of \eqref{ASFF1} that
\begin{equation*}
\begin{split}
O(t)&=\gamma\int_{0}^{t}U(s)\int_{\R}p(\xi)^{2}\e^{-(|\xi|^{2}+\eta)(t-s)}\,\ud\xi\,\ud s
\\
&=\frac{2\sin(\alpha\pi)}{\pi}\int_{0}^{t}\int_{0}^{+\infty}\rho^{2\alpha-1}\e^{-(\rho^{2}+\eta)(t-s)}\,\ud\rho\,U(s)\,\ud s.
\end{split}
\end{equation*}
Now using the fact that $\ds\frac{1}{\Gamma(\alpha)\Gamma(1-\alpha)}=\frac{\sin(\alpha\pi)}{\pi}$ then a simple change of variable leads to the relation \eqref{ASFF2}. This completes the proof.
\end{proof}
Using now Proposition \ref{ASFF3} and relation \eqref{IFF10}, system \eqref{IFF1} may be recast into the following augmented system
\begin{equation}\label{IFF2}
\partial_{t}^{2}u(t)+Au(t)+\gamma\,B\int_{\R}p(\xi)\,\varphi(t,\xi)\,\ud\xi=0,\;t>0, 
\end{equation}
\begin{equation}\label{IFF3}
\partial_{t}\varphi(t,\xi)+(|\xi|^2+\eta)\,\varphi(t,\xi)=p(\xi)\,B^{*}\partial_{t}u(t),\;\xi\in\R,\;t>0,
\end{equation}
\begin{equation}\label{IFF4}
u(0)=u^{0},\;\partial_{t}u(0)=u^{1},\;\varphi(0,\xi)=0,
\end{equation}
where the function $p(\xi)$ and the constant $\gamma$ are given in Proposition \ref{ASFF3}.
\section{Well-posedness}\label{WFF}
In this section, we are interested in showing that system \eqref{IFF1} is well posed in the sense of semigroups. 

Let $V=L^{2}(\R;U)$, we set the Hilbert space $\mathcal{H}=H_{\frac{1}{2}}\times H\times V$ with inner product
$$
\left\langle\left(
\begin{array}{c}
u_{1} 
\\
v_{1} 
\\
\varphi_{1}
\end{array}
\right),\left(
\begin{array}{c}
u_{2} 
\\
v_{2} 
\\
\varphi_{2}
\end{array}
\right) \right\rangle_{\mathcal{H}}=\left\langle A^{\frac{1}{2}}u_{1},A^{\frac{1}{2}}u_{2}\right\rangle_{H}+\left\langle v_{1},v_{2}\right\rangle_{H}+\gamma\int_{\R}\left\langle\varphi_{1}(\xi),\varphi_{2}(\xi)\right\rangle_{U}\,\ud\xi.
$$
If we put $X=\left(
\begin{array}{c}
u 
\\
\partial_{t}u 
\\
\varphi
\end{array}
\right)$ it is clear that \eqref{IFF2}-\eqref{IFF4} can be written as
\begin{equation}\label{IFF8}
X^{\prime}(t)={\mathcal A}X (t),\;\;X(0)=X_{0},
\end{equation}
where $X_{0}=\left(
\begin{array}{c}
u_{0}
\\
u_{1} 
\\
0
\end{array}
\right)$ and $\mathcal{A}:\mathcal{D}(\mathcal{A})\subset\mathcal{H}\rightarrow\mathcal{H}$ is defined by
\begin{equation}\label{IFF6}
\mathcal{A} \left(
\begin{array}{c}
u
\\
v
\\
\varphi
\end{array}
\right)=\left(
\begin{array}{c}
v
\\
\ds -Au-\gamma\,B\int_{\R}p(\xi)\,\varphi(\xi)\,\ud\xi
\\
-(|\xi|^2+\eta)\varphi+p(\xi)B^{*}v
\end{array}
\right),
\end{equation}
with domain
\begin{equation}\label{IFF7}
\begin{split}
\mathcal{D}(\mathcal{A})=\Big\{(u,v,\varphi) \in \mathcal{H}:\; v \in H_{\frac{1}{2}},\;Au +\gamma\,B\int_{\R}p(\xi)\,\varphi(\xi)\,\ud\xi\in H,
\\
|\xi|\varphi\in L^{2}(\R;U),\;-(|\xi|^{2}+\eta)\varphi+p(\xi)B^{*}v\in L^{2}(\R;U)\Big\}.
\end{split}
\end{equation}
Our main result is giving by the following theorem.
\begin{thm}\label{WFF1}
The operator $\mathcal{A}$ defined by \eqref{IFF6} and \eqref{IFF7}, generates a $C_{0}$ semigroup of contractions $e^{t\mathcal{A}}$ in the Hilbert space $\mathcal{H}$.
\end{thm}
\begin{proof}
To prove this result we shall use the Lumer-Phillips' theorem (see \cite[Theorem 4.3]{Pazy}). Since for every $X=(u,v,\varphi)\in\mathcal{D}(\mathcal{A})$ we have
$$
\re\left\langle\mathcal{A}X,X\right\rangle_{\mathcal{H}}=-\gamma\int_{\R}(|\xi|^{2}+\eta)\|\varphi(\xi)\|_{U}^{2}\,\ud\xi\leq 0.
$$
then the operator $\mathcal{A}$ is dissipative. 

Let $\lambda>0$, we prove that the operator $(\lambda I-\mathcal{A})$ is a surjection. In other words, we shall demonstrate that given any triplet $Z=(f,g,h)\in \mathcal{H}$, there is an other triplet $X=(u,v,\varphi)\in\mathcal{D}(\mathcal{A})$ such that $(\lambda I-\mathcal{A})X=Z$, which can be recast as follow
\begin{equation*}
\left\{\begin{array}{l}
v=\lambda u-f,
\\
\ds(\lambda^{2}I+A)u=\lambda f+g-\gamma B\int_{\R}p(\xi)\varphi(\xi)\,\ud \xi,
\\
\ds\varphi(\xi)=\frac{p(\xi)}{|\xi|^{2}+\eta+\lambda}B^{*}v+\frac{h(\xi)}{|\xi|^{2}+\eta+\lambda}.
\end{array}\right.
\end{equation*}

Since $A$ is a non-negative operator then according \cite[Proposition 3.3.5]{tucsnakweinss}, $-A$ is m-dissipative. Thus the operator $(\lambda^{2}+A)$ is a bijection and we have
$$
\|(\lambda^{2}I+A)^{-1}\|_{\mathcal{L}(H)}\leq \frac{1}{\lambda^{2}}.
$$
Let $(u_{n})$, $(v_{n})$ and $(\varphi_{n})$ be three sequences defined by induction as follow
$$
\left\{\begin{array}{l}
u_{0}=(\lambda^{2}I+A)^{-1}(\lambda f+g)\in H_{1}\subset H_{\frac{1}{2}},
\\
v_{0}=-f\in H^{\frac{1}{2}}\subset H,
\\
\ds\varphi_{0}(\xi)=\frac{h(\xi)}{|\xi|^{2}+\eta+\lambda}\in V,
\end{array}\right.
$$
and
$$
\left\{\begin{array}{l}
\ds u_{n+1}=-\gamma(\lambda^{2}I+A)^{-1}B\int_{\R}p(\xi)\varphi_{n}(\xi)\,\ud\xi,
\\
v_{n+1}=\lambda u_{n},
\\
\ds\varphi_{n+1}(\xi)=\frac{p(\xi)}{|\xi|^{2}+\eta+\lambda}B^{*}v_{n}.
\end{array}\right.
$$
We denote the constants $C_{1}$, $C_{2}$ and $C_{3}$ by
$$
\begin{array}{lll}
C_{1}=\|f\|_{H_{-\frac{1}{2}}}+\|g\|_{H_{-\frac{1}{2}}},&C_{2}=\|f\|_{H_{-\frac{1}{2}}},&\ds C_{3}=\left(\int_{\R}\frac{(1+|\xi|)^{2}}{(|\xi|^{2}+\eta+\lambda)^{2}}\,\ud\xi\right)^{\frac{1}{2}}\|h\|_{V},
\end{array}
$$
and we set the constants $K_{1}$ and $K_{2}$ by
$$
\begin{array}{ll}
\ds K_{1}=\|B\|_{\mathcal{L}(U,H_{-\frac{1}{2}})}\left(\int_{\R}\frac{p(\xi)^{2}}{(1+|\xi|)^{2}}\ud\xi\right)^{\frac{1}{2}},&\ds K_{2}=\gamma\|B^{*}\|_{\mathcal{L}(H_{\frac{1}{2}},U)}\left(\int_{\R}\left(\frac{p(\xi)(1+|\xi|)}{|\xi|^{2}+\eta+\lambda}\right)^{2}\ud\xi\right)^{\frac{1}{2}},
\end{array}
$$
which it is clear that they are well defined.

We set the sequences $a_{n}=\|u_{n}\|_{H_{-\frac{1}{2}}}$, $b_{n}=\|v_{n}\|_{H_{-\frac{1}{2}}}$ and $c_{n}=\|\,(1+|\xi|).\varphi_{n}\|_{L^{2}(\R,U)}$. It is clear using the H\"older inequality that 
$$
\begin{array}{lll}
\ds a_{0}\leq\lambda^{-1}C_{1},&b_{0}\leq C_{2},& c_{0}\leq C_{3},
\\
\\
\ds a_{1}\leq\lambda^{-2}C_{3}K_{1},&b_{1}\leq C_{1},& c_{1}\leq C_{2}K_{2},
\\
\\
\ds a_{2}\leq\lambda^{-2}C_{2}K_{1}K_{2},&\ds b_{2}\leq\lambda^{-1}C_{3}K_{1},& c_{2}\leq C_{1}K_{2}.
\end{array}
$$
Using the same arguments we can prove by induction that for all $n\in\N$ we have $u_{n},\,v_{n}\in H_{\frac{1}{2}}$, $\varphi_{n},|\xi|\varphi_{n}\in V$ and
$$
\begin{array}{lll}
\ds a_{3n}\leq\lambda^{-(n+1)}C_{1}K_{1}^{n}K_{2}^{n},&b_{3n}\leq\lambda^{-n}C_{2}K_{1}^{n}K_{2}^{n},& c_{3n}\leq\lambda^{-n}C_{3}K_{1}^{n}K_{2}^{n},
\\
\\
\ds a_{3n+1}\leq\lambda^{-(n+2)}C_{3}K_{1}^{n+1}K_{2}^{n},&b_{3n+1}\leq \lambda^{-n}C_{1}K_{1}^{n}K_{2}^{n},& c_{3n+1}\leq \lambda^{-n}C_{2}K_{1}^{n}K_{2}^{n+1},
\\
\\
\ds a_{3n+2}\leq\lambda^{-(n+2)}C_{2}K_{1}^{n+1}K_{2}^{n+1},&\ds b_{3n+2}\leq\lambda^{-(n+1)}C_{3}K_{1}^{n+1}K_{2}^{n},& c_{3n+2}\leq\lambda^{-n}C_{1}K_{1}^{n}K_{2}^{n+1}.
\end{array}
$$
So that, for $\lambda>0$  large enough the two sums $\ds\sum u_{n}$ and $\ds\sum v_{n}$ converge uniformly in $H_{-\frac{1}{2}}$ and the sum $\ds\sum\varphi_{n}$ converges uniformly in $V$. Therefore, by setting $\ds u=\sum_{n=0}^{+\infty} u_{n}$, $\ds v=\sum_{n=0}^{+\infty}v_{n}$ and $\ds\varphi=\sum_{n=0}^{+\infty}\varphi_{n}$ we find
\begin{equation*}
\begin{split}
u&=u_{0}+\sum_{n=1}^{+\infty}u_{n}=(\lambda^{2}I+A)^{-1}(\lambda f+g)-\gamma\sum_{n=1}^{+\infty}(\lambda^{2}I+A)^{-1}B\int_{\R}p(\xi)\varphi_{n-1}(\xi)\,\ud\xi
\\
&=(\lambda^{2}I+A)^{-1}\left((\lambda f+g)-\gamma B\int_{\R}p(\xi)\sum_{n=0}^{+\infty}\varphi_{n}(\xi)\,\ud\xi\right)
\\
&=(\lambda^{2}I+A)^{-1}(\lambda f+g)-\gamma(\lambda^{2}I+A)^{-1}B\int_{\R}p(\xi)\varphi(\xi)\,\ud\xi.
\end{split}
\end{equation*}
Since $\varphi\in V$ we follow that $u\in H_{\frac{1}{2}}$ and we have $\ds(\lambda^{2}I+A)u=(\lambda f+g)-\gamma B\int_{\R}p(\xi)\varphi(\xi)\,\ud\xi$. By the same way we have
$$
v=\sum_{n=0}^{+\infty}v_{n}=v_{0}+\sum_{n=1}^{+\infty}v_{n}=f+\lambda\sum_{n=1}^{+\infty}u_{n-1}=\lambda u+f
$$
and also
\begin{equation*}
\begin{split}
\varphi(\xi)&=\sum_{n=0}^{+\infty}\varphi_{n}(\xi)=\varphi_{0}(\xi)+\sum_{n=1}^{+\infty}\varphi_{n}(\xi)=\frac{h(\xi)}{|\xi|^{2}+\eta+\lambda}+\frac{p(\xi)}{|\xi|^{2}+\eta+\lambda}B^{*}\sum_{n=1}^{+\infty}v_{n-1}
\\
&=\frac{h(\xi)}{|\xi|^{2}+\eta+\lambda}+\frac{p(\xi)}{|\xi|^{2}+\eta+\lambda}B^{*}v.
\end{split}
\end{equation*}
This prove $v\in H_{\frac{1}{2}}$ and $|\xi|\varphi,\,\varphi\in L^{2}(\R;U)$. Finally, it is clear that $\ds Au+\gamma B\int_{\R}p(\xi)\varphi(\xi)\,\ud\xi\in H$, $-(|\xi|^{2}+\eta)\varphi+p(\xi)B^{*}v\in L^{2}(\R;U)$. Hence, we proved that the operator $(\mathcal{A}-\lambda I)$ is onto. This completes the proof.
\end{proof}
As a consequence of Theorem \ref{WFF1}, the system \eqref{IFF2}-\eqref{IFF4} is well-posed in the energy space $\mathcal{H}$ and we have the following proposition.
\begin{prop}
For $(u^{0},u^{1},0)\in \mathcal{H}$, the problem \eqref{IFF2}-\eqref{IFF4} admits a unique solution 
$$
(u,\partial_{t}u,\varphi)\in\mathcal{C}([0,+\infty);\mathcal{H}).
$$
and for $(u^{0},u^{1},0)\in\mathcal{D}(\mathcal{A})$, the problem \eqref{IFF2}-\eqref{IFF4} admits a unique solution 
$$
(u,\partial_{t}u,\varphi)\in\mathcal{C}([0,+\infty);\mathcal{D}(\mathcal{A}))\cap\mathcal{C}^{1}([0,+\infty);\mathcal{H}).
$$
Moreover, from the density of $\mathcal{D}(\mathcal{A})$ in $\mathcal{H}$ the energy of $(u(t),\varphi(t))$ at time $t\geq 0$ given by
\[
E(t)=\frac{1}{2}\,\left(\|u(t)\|_{H_{\frac{1}{2}}}^{2}+\|\partial_{t}u(t)\|_{H}^{2}+\gamma\int_{\R}\left\|\varphi(t,\xi)\right\|_{U}^{2}\,\ud\xi \right).
\]
decays as follow
\begin{equation}\label{IFF5}
\frac{\ud E}{\ud t}(t)=-\gamma\int_{\R}\left(|\xi|^{2}+\eta\right)\left\|\varphi(t,\xi)\right\|^2_U\,\ud\xi,\;\forall\, t\geq 0.
\end{equation}
\end{prop}
\begin{proof}
Noting that the regularity of the solution of the problem \eqref{IFF2}-\eqref{IFF4} is consequence of the semigroup properties. We have just to prove \eqref{IFF5}. We set 
$$
E_{1}(t)=\frac{1}{2}\left(\|u(t)\|_{H_{\frac{1}{2}}}^{2}+\|\partial_{t}u(t)\|_{H}^{2}\right)\quad\text{and}\quad E_{2}(t)=\frac{\gamma}{2}\left(\int_{\R}\,\|\varphi(t,\xi)\|_{U}^{2}\,\ud \xi\right).
$$
A straightforward calculation gives
$$
\frac{\ud E_{1}}{\ud t}(t)=-\gamma\re\left\langle\int_{\R}p(\xi)\varphi(t,\xi)\,\ud\xi,B^{*}\partial_{t}u(t)\right\rangle_{U},
$$
and
$$
\frac{\ud E_{2}}{\ud t}(t)=\gamma\re\left\langle\int_{\R}p(\xi)\varphi(t,\xi)\,\ud\xi,B^{*}\partial_{t}u(t)\right\rangle_{U}-\gamma\int_{\R}(|\xi|^{2}+\eta).\|\varphi(t,\xi)\|_{U}^{2}\,\ud\xi.
$$
Since $E(t)=E_{1}(t)+E_{2}(t)$ then estimate \eqref{IFF5} holds and this complete the proof.
\end{proof}
\section{Strong stabilization}\label{SSFF}
In this section, we prove that the solutions of system \eqref{IFF8} converge asymptotically to zero. To achieve this result,we shall make use the LaSalle's invariance principle extended to infinite-dimensional systems \cite{walter}. According  to this principle, all solutions of \eqref{IFF8} will asymptotically tend to the maximal invariant subset of the set
$$
\mathcal{I}=\left\{X_{0}\in\mathcal{H}:\; \frac{\ud E}{\ud t}(t)=0\right\}.
$$
Provided that these solutions are pre-compact in $\mathcal{H}$.
\begin{lem}
Let
$$
\mathcal{E}(t)=\frac{1}{2}\left(\|\partial_{t}u\|_{H_{\frac{1}{2}}}^{2}+\|\partial_{t}^{2}u\|_{H}^{2}+\gamma\int_{\R}\,\|\partial_{t}\varphi(\xi)\|_{U}^{2}\,\ud \xi\right).
$$
Then the function $t\longmapsto\mathcal{E}(t)$ is non-increasing along solutions of the system \eqref{IFF8} with initial data are in $\mathcal{D}(\mathcal{A}^{2})$. In particular, we have
\begin{equation}\label{SSFF1}
\frac{\ud \mathcal{E}}{\ud t}(t)=-\gamma\int_{\R}(|\xi|^{2}+\eta).\|\partial_{t}\varphi(t,\xi)\|_{U}^{2}\,\ud\xi.
\end{equation}
\end{lem}
\begin{proof}
If $X_{0}\in\mathcal{D}(\mathcal{A}^{2})$ then the $X(t)=e^{t\mathcal{A}}X_{0}$ is a solution of \eqref{IFF8} with the following regularity
$$
X(t)=\left(\begin{array}{c}
u(t)
\\
\partial_{t}u(t)
\\
\varphi(t)
\end{array}\right)\in\mathcal{C}([0,+\infty[,\mathcal{D}(\mathcal{A}^{2}))\cap\mathcal{C}^{1}([0,+\infty[,\mathcal{D}(\mathcal{A})).
$$
with $\dot{X}(t)=\left(\begin{array}{c}
\partial_{t}u(t)
\\
\partial_{t}^{2}u(t)
\\
\partial_{t}\varphi(t,\xi)
\end{array}\right)=\mathcal{A}X(t)=\mathcal{A}e^{t\mathcal{A}}X_{0}=e^{t\mathcal{A}}\mathcal{A}X_{0}$. And since $\mathcal{A}X_{0}\in\mathcal{D}(\mathcal{A})$ then 
$$
\dot{X}(t)\in\mathcal{C}([0,+\infty[,\mathcal{D}(\mathcal{A}))\cap\mathcal{C}^{1}([0,+\infty[,\mathcal{H}),
$$
then by setting
$$
\mathcal{E}_{1}(t)=\frac{1}{2}\left(\|\partial_{t}u\|_{H_{\frac{1}{2}}}^{2}+\|\partial_{t}^{2}u\|_{H}^{2}\right)\quad\text{and}\quad\mathcal{E}_{2}(t)=\frac{\gamma}{2}\left(\int_{\R}\,\|\partial_{t}\varphi(\xi)\|_{U}^{2}\,\ud \xi\right)
$$
we have
$$
\frac{\ud\mathcal{E}_{1}}{\ud t}(t)=-\gamma\re\left\langle\int_{\R}p(\xi)\partial_{t}\varphi(t,\xi)\,\ud\xi,B^{*}\partial_{t}^{2}u\right\rangle_{U},
$$
and
$$
\frac{\ud\mathcal{E}_{2}}{\ud t}(t)=\gamma\re\left\langle\int_{\R}p(\xi)\partial_{t}\varphi(t,\xi)\,\ud\xi,B^{*}\partial_{t}^{2}u\right\rangle_{U}-\gamma\int_{\R}(|\xi|^{2}+\eta).\|\partial_{t}\varphi(t,\xi)\|_{U}^{2}\,\ud\xi.
$$
So that, by summing the two last expressions we obtain \eqref{SSFF1} and consequently the non-increasing property of $\mathcal{E}(t)$ holds. This complete the proof.
\end{proof}
\begin{lem}\label{SSFF17}
We assume that the only classical solution $u(t)$ (i.e. such that  for all $t\geq 0$ we have $(u(t),\,\partial_{t}u(t))\in H_{\frac{1}{2}}\times H_{\frac{1}{2}}$ and $Au(t)\in H$) of the following system
\begin{equation}\label{SSFF21}
\left\{\begin{array}{l}
\partial_{t}^{2}u(t)+Au(t)=0
\\
B^{*}\partial_{t}u(t)=0.
\end{array}\right.
\end{equation}
is the trivial one, then \eqref{IFF8} admits a unique solution too given by the zero solution.
\end{lem}
\begin{proof}
Let $X=(u,v,\varphi)\in\mathcal{I}$ be a classical solution of \eqref{IFF8}. Then from \eqref{IFF5} we have
$$
\int_{\R}\left(|\xi|^{2}+\eta\right)\left\|\varphi(s,\xi)\right\|^{2}_{U}\,\ud\xi=0.
$$
which imply that
\begin{equation}\label{SSFF2}
\varphi(t,\xi)\equiv 0 \;\text{ in }\;L^{2}(\R;U).
\end{equation}
By using \eqref{SSFF2}, it is clear that system \eqref{IFF8} reduces to the system \eqref{SSFF21}. Then by the assumption made in this lemma we deduce that $u(t)\equiv 0$ for all $t\geq0$. This complete the proof.
\end{proof}
\begin{prop}\label{SSFF18}
Let $X_{0}=(u_{0},v_{0},\varphi_{0})\in\mathcal{D}(\mathcal{A}^{2})$, then the trajectory of $\varphi(t)$, the third component of the solution of \eqref{IFF8}, is pre-compact in $L^{2}(\R;U)$.
\end{prop}
\begin{proof}
Since, for $X_{0}\in\mathcal{D}(\mathcal{A}^{2})$, $\varphi(t)$ is continuous mapping from $[0,+\infty[$ into $L^{2}(\R,U)$, it is therefore sufficient to show that 
$$
\int_{\R}\|\varphi(t,\xi)\|_{U}^{2}\,\ud \xi\,\longrightarrow\,0\quad\text{ as }\quad t\,\nearrow\,+\infty.
$$
From \eqref{IFF5} and \eqref{SSFF1} together with the fact that both $E(t)$ and $\mathcal{E}(t)$ are non-increasing functions we follow
\begin{equation}\label{SSFF3}
\int_{0}^{+\infty}\int_{\R}(|\xi|^{2}+\eta)\|\varphi(t,\xi)\|_{U}^{2}\,\ud\xi\,\ud t<+\infty,
\end{equation}
and
\begin{equation}\label{SSFF4}
\int_{0}^{+\infty}\int_{\R}(|\xi|^{2}+\eta)\|\partial_{t}\varphi(t,\xi)\|_{U}^{2}\,\ud\xi\,\ud t<+\infty.
\end{equation}
The remainder of the proof will be divided in two cases.

\underline{\textbf{Case 1:} $\eta\neq 0$.} Here we get immediately from \eqref{SSFF3} and \eqref{SSFF4} the following relations
\begin{equation}\label{SSFF5}
\int_{0}^{+\infty}\int_{\R}\|\varphi(t,\xi)\|_{U}^{2}\,\ud\xi\,\ud t<+\infty,
\end{equation}
and
\begin{equation}\label{SSFF6}
\int_{0}^{+\infty}\int_{\R}\|\partial_{t}\varphi(t,\xi)\|_{U}^{2}\,\ud\xi\,\ud t<+\infty.
\end{equation}
By using theses relations together with the well know inequality $2\,\re\langle X,Y\rangle\leq \|X\|^{2}+\|Y\|^{2}$ for all $X,\,Y$, we obtain
\begin{equation*}
\begin{split}
\left|\int_{\R}\|\varphi(t,\xi)\|_{U}^{2}\,\ud\xi-\int_{\R}\|\varphi(s,\xi)\|_{U}^{2}\,\ud\xi\right|&=2\,\left|\re\left(\int_{s}^{t}\int_{\R}\langle\partial_{t}\varphi(t,\xi),\varphi(t,\xi)\rangle_{U}\,\ud\xi\,\ud t\right)\right|
\\
&\leq\int_{s}^{t}\int_{\R}\|\partial_{t}\varphi(\xi,t)\|_{U}^{2}+\|\varphi(\xi,t)\|_{U}^{2}\,\ud\xi\,\ud t,
\end{split}
\end{equation*}
then we easily see from \eqref{SSFF5} and \eqref{SSFF6} that
\begin{equation}\label{SSFF7}
\lim_{t\to+\infty}\int_{\R}\|\varphi(t,\xi)\|_{U}^{2}\,\ud\xi\quad\text{ exist and finite.}
\end{equation}
But then \eqref{SSFF5} and \eqref{SSFF7} imply that
$$
\lim_{t\to+\infty}\int_{\R}\|\varphi(t,\xi)\|_{U}^{2}\,\ud\xi=0.
$$

\underline{\textbf{Case 2:} $\eta= 0$.} In this case \eqref{SSFF3} and \eqref{SSFF4} reduce to
\begin{equation}\label{SSFF8}
\int_{0}^{+\infty}\int_{\R}|\xi|^{2}\|\varphi(t,\xi)\|_{U}^{2}\,\ud\xi\,\ud t<+\infty,
\end{equation}
and
\begin{equation}\label{SSFF9}
\int_{0}^{+\infty}\int_{\R}|\xi|^{2}\|\partial_{t}\varphi(t,\xi)\|_{U}^{2}\,\ud\xi\,\ud t<+\infty.
\end{equation}
Again, by using the inequality $2\,\re\langle X,Y\rangle\leq \|X\|^{2}+\|Y\|^{2}$, we have
$$
\lim_{t\to+\infty}\int_{\R}|\xi|^{2}\|\varphi(t,\xi)\|_{U}^{2}\,\ud\xi\quad\text{ exist and finite.}
$$
Thus \eqref{SSFF8} imply that
\begin{equation}\label{SSFF10}
\lim_{t\to+\infty}\int_{\R}|\xi|^{2}\|\varphi(t,\xi)\|_{U}^{2}\,\ud\xi=0.
\end{equation}
Therefore, in view of \eqref{SSFF10}, it is clear that $\ds\int_{\R}\|\varphi(t,\xi)\|_{U}^{2}\,\ud\xi$ will tends to zero as $t$ goes to $+\infty$, if 
\begin{equation}\label{SSFF11}
\lim_{t\to+\infty}\int_{B(0,1)}\|\varphi(t,\xi)\|_{U}^{2}\,\ud\xi=0,
\end{equation}
where $B(0,1)$ is the unit ball in $\R$. Next, we prove \eqref{SSFF11} by using the dominated converges theorem whose conditions of applicability, in the case at hand, are established below:

$\ast)$ By applying Fubini's theorem to both inequality \eqref{SSFF8} and \eqref{SSFF9} we have
$$
\int_{0}^{+\infty}|\xi|^{2}\|\varphi(\xi,t)\|_{U}^{2}\,\ud t<+\infty\;\text{a.e }\xi\in B(0,1)
$$
and
$$
\int_{0}^{+\infty}|\xi|^{2}\|\partial_{t}\varphi(\xi,t)\|_{U}^{2}\,\ud t<+\infty\;\text{a.e }\xi\in B(0,1).
$$
So that, by the same argument that led us to \eqref{SSFF10}, we may conclude that
$$
\lim_{t\to+\infty}|\xi|^{2}\|\varphi(t,\xi)\|_{U}^{2}=0\qquad\text{a.e }\xi\in B(0,1).
$$
Hence, we obtain
\begin{equation}\label{SSFF12}
\lim_{t\to+\infty}\|\varphi(t,\xi)\|_{U}^{2}=0\qquad\text{a.e }\xi\in B(0,1).
\end{equation}

$\ast)$ Now solving \eqref{IFF3}, we have
\begin{equation}\label{SSFF13}
\varphi(t,\xi)=\varphi_{0}(\xi)\e^{-|\xi|^{2}t}+p(\xi)B^{*}\int_{0}^{t}\partial_{t}u(s)\e^{-|\xi|^{2}(t-s)}\,\ud s.
\end{equation}
So that, by applying integration by parts, to the integral in the right hand side of \eqref{SSFF13}, we get
$$
\varphi(t,\xi)=\varphi_{0}(\xi)\e^{-|\xi|^{2}t}+p(\xi)B^{*}[u(t)-u(0)\e^{-|\xi|^{2}t}]-|\xi|^{2}p(\xi)B^{*}\int_{0}^{t}\partial_{t}u(s)\e^{-|\xi|^{2}(t-s)}\,\ud s.
$$
Hence, one gets
\begin{equation*}
\begin{split}
\|\varphi(t,\xi)\|_{U}\leq&\|\varphi_{0}(\xi)\|_{U}+p(\xi)\|B^{*}\|_{\mathcal{L}(H_{\frac{1}{2}},U)}\times
\\
&\left[\|u(t)\|_{H_{\frac{1}{2}}}+\|u(0)\|_{H_{\frac{1}{2}}}+|\xi|^{2}\int_{0}^{t}\|\partial_{t}u(s)\|_{H_{\frac{1}{2}}}\e^{-|\xi|^{2}(t-s)}\,\ud s\right].
\end{split}
\end{equation*}
Also by \eqref{IFF5} we can bound $\|u(t)\|_{H_{\frac{1}{2}}}^{2}\leq E(0)$ and we obtain
\begin{equation}\label{SSFF14}
\|\varphi(t,\xi)\|_{U}^{2}\leq \|\varphi_{0}(\xi)\|_{U}^{2}+p(\xi)^{2}\|B^{*}\|_{\mathcal{L}(H_{\frac{1}{2}},U)}^{2}\left[2E(0)+\mathcal{E}(0)|\xi|^{2}(1-\e^{-|\xi|t})\right].
\end{equation}
Since the right hand side of \eqref{SSFF14} is in $L_{\xi}^{1}(B(0,1))$, therefore by combining \eqref{SSFF12} and \eqref{SSFF14} through the dominated convergence theorem, we get \eqref{SSFF11} the desired result.
\end{proof}
\begin{prop}\label{SSFF19}
We assume that the embedding $H_{\frac{1}{2}}\hookrightarrow H$ is a compact embedding. Let $X_{0}=(u_{0},v_{0},\varphi_{0})\in\mathcal{D}(\mathcal{A}^{2})$, then the trajectory of the pair $(u(t),v(t))$ of the solution of the system \eqref{IFF8} is pre-compact in $H_{\frac{1}{2}}\times H$.
\end{prop}
\begin{proof}
Note that if $X_{0}=(u_{0},v_{0},\varphi_{0})\in\mathcal{D}(\mathcal{A}^{2})$ then $(u(t),v(t))\in H_{1}\times H_{\frac{1}{2}}$. Since that, in view of the assumption made in this proposition it is clear that to prove this result we have just to prove that the quantity $\|u(t)\|_{H_{1}}^{2}+\|v(t)\|_{H_{\frac{1}{2}}}$ is bounded. We solve the differential equation \eqref{IFF3}, we get
\begin{equation}\label{SSFF15}
\begin{split}
\varphi(t,\xi)&=\varphi_{0}(\xi)\e^{-(|\xi|^{2}+\eta)t}+p(\xi)B^{*}\int_{0}^{t}\partial_{t}u(s)\e^{-(|\xi|^{2}+\eta)(t-s)}\,\ud s
\\
&=\varphi_{0}(\xi)\e^{-(|\xi|^{2}+\eta)t}+p(\xi)B^{*}\int_{0}^{t}\partial_{t}u(t-s)\e^{-(|\xi|^{2}+\eta)s}\,\ud s.
\end{split}
\end{equation}
Using the differential equation \eqref{IFF2}, Fubini's theorem and taking account of \eqref{SSFF15} and the fact that $\mathcal{E}(t)$ is bounded by $\mathcal{E}(0)$, we have
\begin{equation}\label{SSFF16}
\begin{split}
\|u(t)\|_{H_{1}}^{2}&+\|v(t)\|_{H_{\frac{1}{2}}}^{2}=\|Au\|_{H}^{2}+\|\partial_{t}u\|_{H^{\frac{1}{2}}}^{2}
\\
&\leq C\Bigg(\mathcal{E}(0)+\|B\|_{\mathcal{L}(U,H_{-\frac{1}{2}})}^{2}\left\|\int_{\R}\!\!\!p(\xi)\varphi_{0}(\xi)\e^{-(|\xi|^{2}+\eta)t}\,\ud\xi\right\|_{U}^{2}
\\
&+\|BB^{*}\|_{\mathcal{L}(H_{\frac{1}{2}},H_{-\frac{1}{2}})}^{2}\left\|\int_{\R}p(\xi)^{2}\int_{0}^{t}\partial_{t}u(t-s)\e^{-(|\xi|^{2}+\eta)s}\,\ud s\,\ud\xi\right\|_{H_{\frac{1}{2}}}^{2}\Bigg)
\\
&\leq C\Bigg(\mathcal{E}(0)+\|B\|_{\mathcal{L}(U,H_{-\frac{1}{2}})}^{2}\int_{0}^{+\infty}\frac{\rho^{2\alpha-1}}{(1+\rho^{2})}\,\ud\rho.\int_{\R}(1+|\xi|^{2})\|\varphi_{0}(\xi)\|_{U}^{2}\,\ud\xi
\\
&+\|BB^{*}\|_{\mathcal{L}(H_{\frac{1}{2}},H_{-\frac{1}{2}})}^{2}\left\|\int_{0}^{t}\int_{0}^{+\infty}\rho^{2\alpha-1}\partial_{t}u(t-s)\e^{-(\rho^{2}+\eta)s}\,\ud\rho\,\ud s\right\|_{H_{\frac{1}{2}}}^{2}\Bigg).
\end{split}
\end{equation}
Now we set 
$$
I=\left\|\int_{0}^{t}\int_{0}^{+\infty}\rho^{2\alpha-1}\partial_{t}u(t-s)\e^{-(\rho^{2}+\eta)s}\,\ud\rho\,\ud s\right\|_{H_{\frac{1}{2}}}^{2},
$$
and to establish our result, it is clear that we have just to prove that $I$ is bounded. To do so we distinguish two cases.

\underline{\textbf{Case 1:} $\eta\neq0$.} Using again Fubini's theorem and the fact that $\mathcal{E}(t)$ is non-increasing function, we obtain
\begin{equation*}
\begin{split}
I&\leq 2\mathcal{E}(0)\left(\int_{0}^{+\infty}\int_{0}^{t}\rho^{2\alpha-1}\e^{-(\rho^{2}+\eta)s}\,\ud s\,\ud\rho\right)^{2}=2E(0)\left(\int_{0}^{+\infty}\frac{\rho^{2\alpha-1}}{\rho^{2}+\eta}\left(1-\e^{-(\rho^{2}+\eta)t}\right)\,\ud\rho\right)^{2}
\\
&\leq 4\mathcal{E}(0)\left(\int_{0}^{+\infty}\frac{\rho^{2\alpha-1}}{\rho^{2}+\eta}\,\ud\rho\right)^{2}<+\infty.
\end{split}
\end{equation*}
which prove that $I$ bounded.

\underline{\textbf{Case 2:} $\eta=0$.} It is clear that according to the first case that the problem of the boundedness of $I$ is reduces to the boundedness of the following integral
$$
I_{0}=\left\|\int_{1}^{t}\int_{0}^{1}\rho^{2\alpha-1}\partial_{t}u(t-s)\e^{-\rho^{2}s}\,\ud\rho\,\ud s\right\|_{H_{\frac{1}{2}}}^{2},
$$
where we can suppose that $t\geq 1$. Integrating by parts with respect to the $s$ variable and using again the fact that $E(t)$ is non-increasing function, we have
\begin{equation*}
\begin{split}
I_{0}&\leq 2\left(\left\|\int_{0}^{1}\rho^{2\alpha-1}\left(\e^{-t\rho^{2}}u(0)-\e^{-\rho^{2}}u(t-1)\right)\,\ud\rho\right\|_{H_{\frac{1}{2}}}^{2}+\left\|\int_{0}^{1}\rho^{2\alpha+1}\int_{1}^{t}\e^{-s\rho^{2}}u(t-s)\,\ud s\,\ud\rho\right\|_{H_{\frac{1}{2}}}^{2}\right)
\\
&\leq C\left(E(0)\left(\int_{0}^{1}\rho^{2\alpha-1}\ud\rho\right)^{2}+E(0)\left(\int_{0}^{1}\rho^{2\alpha-1}(\e^{-\rho^{2}}-\e^{-t\rho^{2}})\,\ud\rho\right)^{2}\right)\leq C\,\!E(0).
\end{split}
\end{equation*}
This prove the expected estimate and end the proof.
\end{proof}
\begin{thm}\label{SSFF20}
Assuming that the embedding $H_{\frac{1}{2}}\hookrightarrow H$ is compact and that system \eqref{SSFF21} admits a unique solution $u$, such that for all $t\geq 0$ we have $(u(t),\,\partial_{t} u(t))\in H_{\frac{1}{2}}\times H_{\frac{1}{2}}$ and $Au(t)\in H$, is the trivial solution, then the semigroup $e^{t\mathcal{A}}$ is strongly stable, it means that for any initial data $X_{0}\in\mathcal{H}$,
$$
\|\e^{t\mathcal{A}}X_{0}\|_{\mathcal{H}}\,\longrightarrow\,0\quad\text{ as }\quad t\longrightarrow +\infty.
$$
\end{thm}
\begin{proof}
For $X_{0}\in\mathcal{D}(\mathcal{A}^{2})$, the theorem is a direct consequence of Lemma \ref{SSFF17}, Propositions \ref{SSFF18} and \ref{SSFF18} and the LaSalle's invariance principle. Finally, since $\mathcal{D}(\mathcal{A}^{2})$ is dense in  $\mathcal{H}$ this result carries over all $X_{0}\in\mathcal{H}$. 
\end{proof}
\section{Lack of uniform stabilization}\label{LUSFF}
In this section we shall prove that system is not uniformly exponentially stable.
\begin{lem}\label{LUSFF8}
Let $\omega\in\R^{*}$ then for any fixed $\eta>0$ and $0<\alpha<1$ we have
\begin{equation}\label{LUSFF1}
\int_{0}^{+\infty}\frac{\rho^{2\alpha-1}}{\rho^{2}+\eta+i\omega}\,\ud\rho=\left\{\begin{array}{ll}
\ds\frac{-\pi(1+\e^{-2i\alpha\pi})}{2(\eta^{2}+\omega^{2})^{\frac{1-\alpha}{2}}\sin(2\alpha\pi)}\e^{2i(\alpha-1)\theta}&\text{if }\ds\alpha\neq \frac{1}{2}
\\
\ds\frac{\pi}{2(\eta^{2}+\omega^{2})^{\frac{1}{4}}\e^{i\theta}}&\text{if }\ds\alpha=\frac{1}{2},
\end{array}\right.
\end{equation}
where  we have denoted by $\ds\theta=\arccos\left(-\frac{\sqrt{\frac{\sqrt{\eta^{2}+\omega^{2}}-\eta}{2}}}{(\eta^{2}+\omega^{2})^{\frac{1}{4}}}\right)$.
\end{lem}
\begin{proof}
The two case are proven as follow:

\underline{\textbf{Case 1:} $\ds\eta\neq\frac{1}{2}$.} In this case the integral can be evaluated using the method of residues. Integrating along the positive oriented contour depicted in Figure \ref{fig1FF}.
\begin{figure}[htbp]
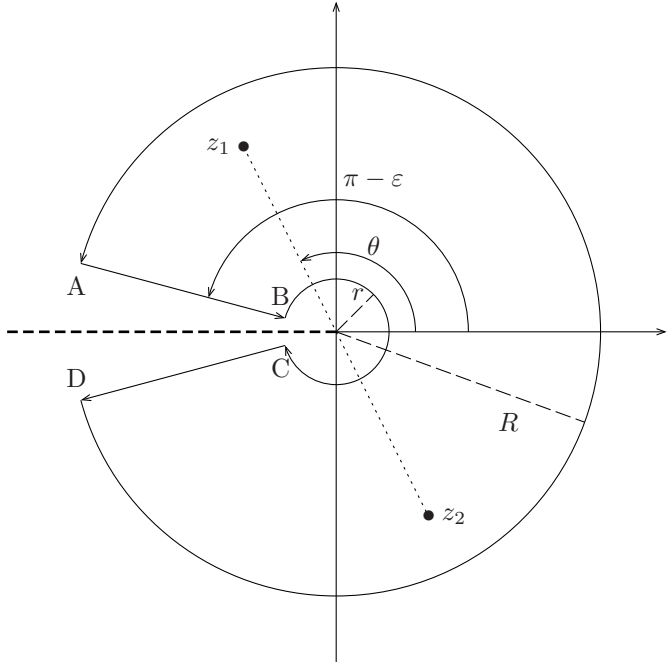

\figinit{pt}
\figpt 0:(0,0)
\figpt 1:(100,0)
\figptrot 2 :A = 1 /0, 165/
\figptrot 3 :D = 1 /0, 195/
\figpt 4:(20,0)
\figptrot 5 :B = 4 /0, 165/
\figptrot 6 :C = 4 /0, 195/
\figpt 7:(-35,70)
\figpt 8:(35,-70)
\figptrot 9 : = 1 /0, 340/
\figptrot 10 : = 4 /0,45/
\figpt 11:(-125,0)
\psbeginfig{}
\psset(dash=5)
\psline[7,8]
\psset(dash=2)
\psline[0,9]
\psline[0,10]
\psset(width=1,dash=3)
\psline[0,11]
\psset(fillmode=yes,width=\defaultwidth,dash=\defaultdash)  \psset arrowhead(length=4)
\psarrowcircP 0;-20[5,6]
\psarrowcircP 0;100[6,5]
\psarrowcircP 0;30[1,7]
\psarrowcircP 0;50[1,2]
\psarrow[2,5]
\psarrow[6,3]
\psaxes 0(0,125,-125,125)
\psendfig
\figvisu{\figBoxA}{}{
\figwriten 3,5:(5)
\figwrites 2,6:(5)
\figwritew 9:$R$(25)
\figwritew 10:$r$(3.5)
\figwriten 10:$\theta$(15)
\figwriten 10:$\pi-\varepsilon$(40)
\figsetmark{$\bullet$}\figwritew 7:$z_{1}$(5)
\figsetmark{$\bullet$}\figwritee 8:$z_{2}$(5)
}
\centerline{\box\figBoxA}
\caption{Contour for evaluating the integral $\ds\int_{0}^{+\infty}\frac{\rho^{2\alpha-1}}{\rho^{2}+\eta+i\omega}\,\ud \rho$.}
\label{fig1FF}
\end{figure}
We set the function
$$
f(z)=\frac{z^{2\alpha-1}}{z^{2}+\eta+i\omega},\qquad\forall\,z\in\mathbb{C}\setminus\R_{-},
$$
whose poles are $z_{1}=(\eta^{2}+\omega^{2})^{\frac{1}{4}}\e^{i\theta}$, $z_{2}=(\eta^{2}+\omega^{2})^{\frac{1}{4}}\e^{i(\theta-\pi)}$ and eventually $z_{0}=0$ (see Figure \ref{fig1FF}). Clearly, we have
\begin{equation}\label{LUSFF2}
|zf(z)|\leq \frac{|z|^{2\alpha}}{\left||z|^{2}-(\eta^{2}+\omega^{2})^{\frac{1}{2}}\right|}
\end{equation}
which imply that
$$
\lim_{z\to 0}zf(z)=0,\qquad \lim_{|z|\to +\infty}zf(z)=0.
$$
Then by Jordon's lemmas we follow 
\begin{equation}\label{LUSFF3}
\lim_{r\to 0}\int_{\gamma_{r}}f(z)\,\ud z=0
\end{equation}
and
\begin{equation}\label{LUSFF4}
\lim_{R\to +\infty}\int_{\gamma_{R}}f(z)\,\ud z=0
\end{equation}
where $\gamma_{r}=r\e^{-it}$ and $\gamma_{R}=r\e^{it}$ for $t\in[-\pi+\varepsilon,\pi-\varepsilon]$ (see Figure \ref{fig1FF}).
\\
On the segment $[AB]$ one has $z=\gamma_{AB}(t)=\left[(1-t)R+rt\right]\e^{i(\pi-\varepsilon)}$ for $t\in[0,1]$ (see Figure \ref{fig1FF}), whence by Lebegue dominated convergence theorem we have
\begin{equation}\label{LUSFF5}
\int_{\gamma_{AB}}f(z)\,\ud z=\e^{i(\varepsilon+2\alpha\pi)}\int_{r}^{R}\frac{\rho^{2\alpha-1}}{\rho^{2}+\eta+i\omega}\,\ud\rho\,\longrightarrow\,-\e^{2i\alpha\pi}\int_{r}^{R}\frac{\rho^{2\alpha-1}}{\rho^{2}+\eta+i\omega}\,\ud\rho\;\text{ as }\;\varepsilon\searrow 0.
\end{equation}
On the segment $[CD]$ one has $z=\gamma_{CD}(t)=t\e^{i(-\pi+\varepsilon)}$ for $t\in[r,R]$ (see Figure \ref{fig1FF}), whence again by Lebegue dominated convergence theorem we have
\begin{equation}\label{LUSFF6}
\int_{\gamma_{CD}}f(z)\,\ud z=\int_{r}^{R}\frac{\rho^{2\alpha-1}\e^{2i\alpha(\varepsilon-\pi)}}{\rho^{2}\e^{2i(\varepsilon-\pi)}+\eta+i\omega}\,\ud\rho\,\longrightarrow\,\e^{-2i\alpha\pi}\int_{r}^{R}\frac{\rho^{2\alpha-1}}{\rho^{2}+\eta+i\omega}\,\ud\rho\;\text{ as }\;\varepsilon\searrow 0.
\end{equation}
By summing \eqref{LUSFF3}-\eqref{LUSFF6} and taking the limits as $r\searrow 0$ and $R\nearrow +\infty$, the method of residues leads to
\begin{equation*}
\begin{split}
\frac{\e^{-2i\alpha\pi}-\e^{2i\alpha\pi}}{2i\pi}\int_{0}^{+\infty}\frac{\rho^{2\alpha-1}}{\rho^{2}+\eta+i\omega}\,\ud\rho&=\underset{z=z_{1},z_{2}}{\mathrm{Res}}[f(z)]=\frac{z_{1}^{2\alpha-2}+z_{2}^{2\alpha-2}}{2}=\frac{\e^{2i(\alpha-1)\theta}[1+\e^{-2i\alpha\pi}]}{2(\eta^{2}+\omega^{2})^{\frac{1-\alpha}{2}}}
\end{split}
\end{equation*}
which leads to the second line of \eqref{LUSFF1}.

\underline{\textbf{Case 2:} $\ds\eta=\frac{1}{2}$.} Since $z_{1}$ and $z_{2}$ are the unique poles of $f$ then we can write
\begin{equation*}
\begin{split}
\int_{0}^{+\infty}\frac{\ud\rho}{\rho^{2}+\eta+i\omega}&=\frac{1}{2z_{1}}\int_{0}^{+\infty}\frac{\rho-\overline{z}_{1}}{\rho^{2}-2\re(z_{1})\rho+|z_{1}|^{2}}-\frac{\rho-\overline{z}_{2}}{\rho^{2}-2\re(z_{2})\rho+|z_{2}|^{2}}\,\ud\rho
\\
&=\frac{1}{2z_{1}}\int_{0}^{+\infty}\frac{\rho-\overline{z}_{1}}{(\rho^{2}-\re(z_{1}))^{2}+\im(z_{1})^{2}}-\frac{\rho-\overline{z}_{2}}{(\rho^{2}-\re(z_{2}))^{2}+\im(z_{2})^{2}}\,\ud\rho.
\end{split}
\end{equation*}
A straightforward calculation leads to
\begin{equation*}
\int_{0}^{+\infty}\frac{\ud\rho}{\rho^{2}+\eta+i\omega}=\frac{\pi}{2z_{1}},
\end{equation*}
which leads to the first line of \eqref{LUSFF1}. And this finish the proof.
\end{proof}
Let assume that $H$ is an infinite dimensional Hilbert space such that the imbedding $H_{\frac{1}{2}}\hookrightarrow H$ is compact. Since $A$ is a strictly positive operator with compact resolvent then there exist a sequence of eigenvalues $i\omega_{n}$ corresponding to the orthonormal base of the eigenfunctions $\phi_{n}=\left(\begin{array}{c}
\frac{u_{n}}{i\omega_{n}}
\\
u_{n}
\end{array}\right)$ of the operator $\mathcal{A}_{0}=\left(\begin{array}{cc}
0&I
\\
-A&0
\end{array}\right)$ such that $\ds\lim_{n\to+\infty}|\omega_{n}|=+\infty$ where $u_{n}\in H_{\frac{1}{2}}$.
\begin{thm}\label{LUSFF9}
Under the above assumptions we have
\begin{enumerate}
	\item If $B\in\mathcal{L}(U,H)$ the semigroup $e^{t\mathcal{A}}$ is not exponentially stable in the Hilbert space $\mathcal{H}$.
	\item If $B\in\mathcal{L}(U,H_{-\frac{1}{2}})$ the semigroup $e^{t\mathcal{A}}$ is not exponentially stable in the Hilbert space $\mathcal{H}$ providing one of the following statements
	\begin{itemize}
		\item[i)] For some $n\in\N$, $B^{*}u_{n}=0$ (in this case the semigroup $e^{t\mathcal{A}}$ is not even strongly stable and the energy is conserved for some data).
		\item[ii)] For all $n\in\N,$ $B^{*}u_{n}\neq 0$ such that there exists a sub-sequence of $(u_{n})$ verifying $BB^{*}u_{n}\in H$ and $\|BB^{*}u_{n}\|_{H}\leq C$ for some $C>0$ and for every $n\in\N$.
	\end{itemize}
\end{enumerate}
\end{thm}
\begin{proof}
To prove this theorem we shall use the frequency theorem method. We recall that a bounded C$_0$ semigroup generated by an operator $\mathcal{A}$ is exponentially stable if and only if $i\R\cap\sigma(\mathcal{A})=\emptyset$ and satisfies the following identity
$$
\limsup_{\omega\in\R, |\omega| \rightarrow + \infty}\|(i\omega I-\mathcal{A})^{-1}\|_{\mathcal{L}(\mathcal{H})}<+\infty.
$$ 
We distinguish now two cases.

\underline{\textbf{Case 1:} $B^{*}u_{n}=0$ for some $n\in\N$.} It is clear in this case that for a such $n\in\N$ we have $X_{n}=\left(\begin{array}{c}
\frac{u_{n}}{i\omega_{n}}
\\
u_{n}
\\
0
\end{array}\right)\in\mathcal{D}(\mathcal{A})$ and $(i\omega_{n}I-\mathcal{A})X_{n}=0$ which proves that $X_{n}$ is an eigenfunction corresponding to the eigenvalue $i\omega_{n}$. Thus, the semigroup $e^{t\mathcal{A}}$ is not uniformly stable.

\underline{\textbf{Case 2:} $B^{*}u_{n}\neq 0$ for all $n\in\N$.} In this part we shall prove a general result then given in the theorem. In fact, we will show that the following resolvent estimate 
\begin{equation}\label{LUSFF7}
\limsup_{\omega\in\R, |\omega|\rightarrow+\infty}\|\omega^{\alpha-1+\varepsilon}(i\omega I-\mathcal{A})^{-1}\|_{\mathcal{L}(\mathcal{H})}<+\infty
\end{equation}
is not even satisfied, for $\varepsilon>0$ small.

Let $\ds\varphi_{n}(\xi)=\frac{p(\xi)}{|\xi|^{2}+\eta+i\omega_{n}}B^{*}u_{n}$ and $X_{n}=\left(\begin{array}{c}
\frac{u_{n}}{i\omega_{n}}
\\
u_{n}
\\
\varphi_{n}
\end{array}\right)$. It is clear that the integrals 
$$
\int_{\R}\frac{p(\xi)^{2}}{|\xi|^{2}+\eta+i\omega_{n}}\,\ud\xi \qquad \text{and} \qquad \int_{\R}\frac{|\xi|^{2}p(\xi)^{2}}{(|\xi|^{2}+\eta)^{2}+\omega_{n}^{2}}\,\ud\xi
$$
are well defined then $|\xi|\varphi_{n}\in L^{2}(\R;U)$ and by the assumption made in this theorem we have $\ds u_{n}+\gamma\,B\int_{\R}p(\xi)\varphi_{n}(\xi)\,\ud\xi\in H$. Besides, since we have
$$
\int_{\R}\|p(\xi)B^{*}u_{n}-(|\xi|^{2}+\eta)\varphi_{n}\|_{U}^{2}\,\ud\xi=\omega_{n}^{2}\|B^{*}u_{n}\|_{U}^{2}\int_{\R}\frac{p(\xi)^{2}}{(|\xi|^{2}+\eta)^{2}+\omega_{n}^{2}}\ud\xi,
$$
then $p(\xi)B^{*}u_{n}-(|\xi|^{2}+\eta)\varphi_{n}\in L^{2}(\R;U)$ and this shows that $X_{n}\in\mathcal{D}(\mathcal{A})$.

We set now $Y_{n}=\left(\begin{array}{c}
f_{n}
\\
g_{n}
\\
h_{n}
\end{array}\right)\in\mathcal{H}$ such that
$$
(i\omega_{n}I-\mathcal{A})X_{n}=Y_{n}.
$$
Since we have $f_{n}=h_{n}=0$ and
$$
g_{n}=\gamma \,BB^{*}u_{n}\,\int_{\R}\frac{p(\xi)^{2}}{|\xi|^{2}+\eta+i\omega_{n}}\,\ud\xi
$$
We set now
$$
\kappa_{n}=\left\{\begin{array}{cl}
\ds\e^{-i\theta_{n}}&\text{if }\ds\alpha=\frac{1}{2}
\\
\ds-\frac{\pi(1+\e^{-2i\alpha\pi})}{2\cos(\alpha\pi)}\e^{2i(\alpha-1)\theta_{n}}&\text{if }\ds\alpha\neq \frac{1}{2},
\end{array}\right.
$$
where $\ds\theta_{n}=\arccos\left(-\frac{\sqrt{\frac{\sqrt{\eta^{2}+\omega_{n}^{2}}-\eta}{2}}}{(\eta^{2}+\omega_{n}^{2})^{\frac{1}{4}}}\right)$. According to Lemma \ref{LUSFF8}, the function $g_{n}$ can be written as follow
$$
g_{n}=\frac{\kappa_{n}}{(\eta_{n}^{2}+\omega_{n}^{2})^{\frac{1-\alpha}{2}}} BB^{*}u_{n}.
$$
then using the assumption made on the boundedness of the sequence $\left(\|BB^{*}u_n\|_{H}\right)$ we follow that
$$
\omega_{n}^{1-\alpha-\varepsilon}\|g_{n}\|_{H}\leq \frac{C\omega_{n}^{1-\alpha-\varepsilon}}{(\eta^{2}+\omega_{n}^{2})^{\frac{1-\alpha}{2}}}\|BB^{*}u_n\|_{H}\,\longrightarrow\,0\;\text{ as }n\nearrow+\infty.
$$
Hence, by assuming that the imaginary axis is a subset of the resolvent set , we follow
\begin{equation*}
\begin{split}
\limsup_{\omega\in\R, |\omega| \rightarrow + \infty}\|\omega^{\alpha-1+\varepsilon}(i\omega I-\mathcal{A})^{-1}\|_{\mathcal{L}(\mathcal{H})}&\geq\sup_{n\in\N}\|\omega_{n}^{\alpha-1+\varepsilon}(i\omega_{n} I-\mathcal{A})^{-1}\|_{\mathcal{L}(\mathcal{H})}
\\
&\geq\sup_{n\in\N}\omega_{n}^{\alpha-1+\varepsilon}\frac{\|(i\omega_{n}I-\mathcal{A})^{-1}(Y_{n})\|_{\mathcal{H}}}{\|Y_{n}\|_{\mathcal{H}}}
\\
&\geq\lim_{n\to+\infty}\omega_{n}^{\alpha-1+\varepsilon}\frac{\|X_{n}\|_{\mathcal{H}}}{\|Y_{n}\|_{\mathcal{H}}}\geq\lim_{n\to+\infty}\frac{\omega_{n}^{\alpha-1+\varepsilon}}{\|g_{n}\|_{H}}=+\infty.
\end{split}
\end{equation*}
Thus, \eqref{LUSFF7} is not satisfied. So that, the semigroup $e^{t\mathcal{A}}$ is not exponentially stable.
\end{proof}
\begin{rem}
In the infinite dimensional case and when the embedding $H_{\frac{1}{2}}\hookrightarrow H$ is compact then under the assumptions made by Theorem \ref{LUSFF9} the prove of the previous theorem shows that the semigroup $e^{t\mathcal{A}}$ is at least dissipating over the time as $t^{-\frac{1}{1-\alpha}}$. In the next section we will show under some assumptions that the semigroup $e^{t\mathcal{A}}$ is decreasing over the time as $t^{-\frac{1}{1-\alpha}}$. Hence this proves a sharp decay rate on the energy of the system \eqref{IFF1}.
\end{rem}
\section{Non-uniform stabilization}\label{NUSFF}
This section is devoted to study the non uniform stabilization of system \eqref{IFF1}-\eqref{IFF3}. Under some assumptions on the behavior of an auxiliary dissipative operator whose dissipation is generated by the classical $BB^{*}$ operator we prove a polynomial decay result for the system \eqref{IFF1}-\eqref{IFF3}. For this purpose we will use a frequency domain approach.
\begin{prop}
Assume that $\eta=0$, then the operator $-\mathcal{A}$ is not onto and consequently $0\in\sigma(\mathcal{A})$.
\end{prop}
\begin{proof}
Let $Y=(0,0,h(\xi))\in\mathcal{H}$ and assume that there exists $X=(u,v,\varphi)\in\mathcal{D}(\mathcal{A})$ such that
$$
-\mathcal{A}X=Y.
$$
It follows that $v=0$, $\ds\varphi(\xi)=\frac{h(\xi)}{|\xi|^{2}}$ and 
$\ds Au+\gamma B\int_{\R}\frac{p(\xi)h(\xi)}{|\xi|^{2}}\,\ud\xi =0$. 
Let $\psi\in U$ such that $\psi\neq 0$ and we set $\ds h(\xi)=\frac{1}{(1+|\xi|)}\psi$. 
It is clear that $h\in L^{2}(\R;U)$. However, $\varphi\notin L^{2}(\R;U)$. Thus, the operator $-\mathcal{A}$ is not onto. This complete the proof.
\end{proof}
\begin{lem}\label{NUSFF1}
Let $\omega\in\R^{*}$ then for any fixed $\eta>0$ and $0<\alpha<1$ we have
\begin{equation}\label{NUSFF2}
\int_{0}^{+\infty}\frac{\rho^{2\alpha-1}}{(\rho^{2}+\eta)^{2}+\omega^{2}}\,\ud\rho=\left\{\begin{array}{ll}
\ds\frac{\sin(2(\alpha-1)(\pi-\phi))-\sin(2(\alpha-1)\phi)}{\sin(2\alpha\pi)\sin(2\phi)(\eta^{2}+\omega^{2})^{1-\frac{\alpha}{2}}}&\text{if }\ds\alpha\neq\frac{1}{2}
\\
\ds\frac{3(2\pi-\phi)}{8(\eta^{2}+\omega^{2})^{\frac{3}{4}}}&\text{if }\ds\alpha=\frac{1}{2},
\end{array}\right.
\end{equation}
where we have denoted by $\ds\phi=\arccos\left(\frac{\sqrt{\frac{\sqrt{\eta^{2}+\omega^{2}}-\eta}{2}}}{(\eta^{2}+\omega^{2})^{\frac{1}{4}}}\right)$.
\end{lem}
\begin{proof}
This prove is the same as the one of Lemma \ref{LUSFF8}. By Keeping the same notations here we just sketch the proof.

\underline{\textbf{Case 1:} $\eta\neq \frac{1}{2}$}. We set the complex function
$$
f(z)=\frac{z^{2\alpha-1}}{(z^{2}+\eta)^{2}+\omega^{2}},\qquad\forall\, z\in\mathbb{C}\setminus\R_{-},
$$
whose poles are $z_{1}^{\pm}=(\eta^{2}+\omega^{2})^{\frac{1}{4}}\e^{\pm i\phi}$, $z_{2}^{\pm}=(\eta^{2}+\omega^{2})^{\frac{1}{4}}\e^{\pm i(\phi-\pi)}$ and eventually $z_{3}=0$. Using the same arguments as Lemma \ref{LUSFF8} we can show that
$$
\int_{\gamma_{r}}f(z)\,\ud z\,\longrightarrow\, 0\quad \text{as}\quad r\searrow 0,
$$
and
$$
\int_{\gamma_{R}}f(z)\,\ud z\,\longrightarrow\, 0\quad \text{as}\quad R\nearrow +\infty.
$$
where on the segments $[AB]$ and $[CD]$ we have
$$
\int_{\gamma_{AB}}f(z)\,\ud z=\int_{r}^{R}\frac{-\e^{2i\alpha(\pi-\epsilon)}\rho^{2\alpha-1}}{(\rho^{2}\e^{2i\epsilon}+\eta)^{2}+\omega^{2}}\,\ud\rho\,\longrightarrow\,\int_{r}^{R}\frac{-\e^{2i\alpha\pi}\rho^{2\alpha-1}}{(\rho^{2}+\eta)^{2}+\omega^{2}}\,\ud\rho\quad \text{as}\quad \varepsilon\searrow 0,
$$
and
$$
\int_{\gamma_{CD}}f(z)\,\ud z=\int_{r}^{R}\frac{\e^{2i\alpha(\epsilon-\pi)}\rho^{2\alpha-1}}{(\rho^{2}\e^{2i(\epsilon-\pi)}+\eta)^{2}+\omega^{2}}\,\ud\rho\,\longrightarrow\,\int_{r}^{R}\frac{\e^{-2i\alpha\pi}\rho^{2\alpha-1}}{(\rho^{2}+\eta)^{2}+\omega^{2}}\,\ud\rho\quad \text{as}\quad \varepsilon\searrow 0,
$$
Summing all these integrals and applying the residues theorem we obtain
$$
\int_{0}^{+\infty}\frac{-\sin(2\alpha\pi)\rho^{2\alpha-1}}{(\rho^{2}+\eta)^{2}+\omega^{2}}\,\ud\rho=\underset{z=z_{1}^{\pm},z_{2}^{\pm}}{\mathrm{Res}}[f(z)]=\frac{(\eta+\omega^{2})^{\frac{\alpha}{2}}(\sin(2(\alpha-1)(\pi-\phi))-\sin(2(\alpha-1)\phi))}{2\cos(\phi)}
$$
which leads obviously to the first line of \eqref{NUSFF2}.

\underline{\textbf{Case 2:} $\eta=\frac{1}{2}$}. In this case we have just to remark that
\begin{equation*}
\begin{split}
\frac{1}{(\rho^{2}+\eta)^{2}+\omega^{2}}&=\frac{1}{8\tau^{3}\cos(\phi)}\left[\frac{\rho+\tau\cos(\phi)}{\rho^{2}+2\tau\cos(\phi)\rho+\tau^{2}}-[\frac{\rho-\tau\cos(\phi)}{\rho^{2}-2\tau\cos(\phi)\rho+\tau^{2}}\right]
\\
&+6\tau\cos(\phi)\left[\frac{1}{\rho^{2}+2\tau\cos(\phi)\rho+\tau^{2}}+\frac{1}{\rho^{2}-2\tau\cos(\phi)\rho+\tau^{2}}\right]
\end{split}
\end{equation*}
where we have denoted by $\tau=(\eta^{2}+\omega^{2})^{\frac{1}{4}}$.
\end{proof}

Let's define now $\mathcal{H}_{0}=H_{\frac{1}{2}}\times H$ and let's consider the operator $\mathcal{A}_{0}:\mathcal{D}(\mathcal{A}_{0})\subset \mathcal{H}_0 \longrightarrow\mathcal{H}_{0}$ defined by
$$
\mathcal{A}_{0}=\left(\begin{array}{rc}
0&I
\\
-A&-BB^{*}
\end{array}\right)
$$
with domain
$$
\mathcal{D}(\mathcal{A}_{0})=\left\{(w,v)\in\mathcal{H}_{0}:\;v\in H_{\frac{1}{2}},\;Aw+BB^{*}v\in H\right\}.
$$ 
\begin{prop}
The operator $\mathcal{A}_{0}$ generates a $C_{0}$ semigroup of contractions in the Hilbert space $\mathcal{H}_{0}$. Moreover, the following auxiliary problem
\begin{equation}\label{NUSFF28}
\left\{\begin{array}{l}
\partial_{t}^{2} w(t)+Aw+BB^{*}\partial_{t}w(t)=0
\\
w(0)=w^{0},\;\partial_{t}w(0)=w^{1}.
\end{array}\right.
\end{equation}
admits a unique solution $w(t,x)$ in such a way that if $(w^{0},w^{1})\in\mathcal{D}(\mathcal{A}_{0})$ the solution $w(t,x)$ of \eqref{NUSFF28} verifying the following regularity
$$
(w,\partial_{t}w)\in\mathcal{C}([0,+\infty);\mathcal{D}(\mathcal{A}_{0}))\cap\mathcal{C}^{1}([0,+\infty);\mathcal{H}_{0}).
$$
and when $(w^{0},w^{1})\in\mathcal{H}_{0}$, we have 
$$
(w,\partial_{t}w)\in\mathcal{C}([0,+\infty);\mathcal{H}_{0}).
$$
The energy of the system \eqref{NUSFF28} defined as follow
$$
E_{0}(t)=\frac{1}{2}\left(\|\partial_{t} w(t)\|_{H}^{2}+\|w(t)\|_{H_{\frac{1}{2}}}^{2}\right),
$$
is decreasing over the time in particular we have
\begin{equation}\label{NUSFF30}
\frac{\ud E_{0}}{\ud t}(t)=-\|B^{*}\partial_{t} w(t)\|_{U}^{2}.
\end{equation}
\end{prop}
\begin{proof}
To show that $\mathcal{A}_{0}$ generates a $C_{0}$ semigroup of contractions we have to prove according to Lumer-Phillips' theorem (see \cite[Theorem 4.3]{Pazy}) that $\mathcal{A}_{0}$ is m-dissipative. First, let $(w,v)\in\mathcal{D}(\mathcal{A}_{0})$ then we have
$$
\re\left\langle \mathcal{A}_{0}\left(\begin{array}{c}
w
\\
v
\end{array}\right),\left(\begin{array}{c}
w
\\
v
\end{array}\right)\right\rangle_{\mathcal{H}_{0}}=-\|B^{*}v\|_{U}^{2}\leq 0,
$$
which proves that $\mathcal{A}_{0}$ is dissipative. It remind now to prove that the range of $I-\mathcal{A}_{0}$ is 
$  \mathcal{H}_0 $. 
For this purpose we let $(f,g)\in\mathcal{H}_{0}$ and we look for a couple $(w,v)\in\mathcal{D}(\mathcal{A}_{0})$ such that 
$$
(I-\mathcal{A}_{0})\left(\begin{array}{c}
w
\\
v
\end{array}\right)=\left(\begin{array}{c}
f
\\
g
\end{array}\right),
$$
or equivalently,
\begin{equation}\label{NUSFF29}
\left\{\begin{array}{l}
v=w+f
\\
Aw+w+BB^{*}w=g-f-BB^{*}f.
\end{array}\right.
\end{equation}
We consider now the following bilinear form on $H_{\frac{1}{2}}\times H_{\frac{1}{2}}$ defined by
$$
L(w,\psi)=\langle w,\psi\rangle_{H_{\frac{1}{2}}}+\langle w,\psi\rangle_{H}+\langle B^{*}w,B^{*}\psi\rangle_{U}.
$$
It is clear that $L$ is continuous and coercive form on $H_{\frac{1}{2}}\times H_{\frac{1}{2}}$ therefore according to 
Lax-Milgram theorem's there exist a unique $w\in H_{\frac{1}{2}}$ such that
$$
L(w,\psi)=\langle g-f,\psi\rangle_{H}-\langle B^{*}f,B^{*}\psi\rangle_{U},\quad \forall\,\psi\in H_{\frac{1}{2}}.
$$
Equivalently, this can be written as follows
$$
\langle Aw+BB^{*}(w+f),\psi\rangle_{H_{-\frac{1}{2}}\times H_{\frac{1}{2}}}=\langle g-f-w,\psi\rangle_{H},\quad \forall\,\psi\in H_{\frac{1}{2}}.
$$
In another words $Aw+BB^{*}(w+f)\in H$ and we have $Aw+w+BB^{*}w=g-f-BB^{*}f$. Since $v=w+f$ then $v\in H^{\frac{1}{2}}$. Hence, system \eqref{NUSFF29} admits a unique solution $(w,v)\in\mathcal{D}(\mathcal{A}_{0})$. Thus, the operator $\mathcal{A}_{0}$ is m-dissipative and consequently the existence and the uniqueness of the solution of problem \eqref{NUSFF28} holds with regularity as described above. Finally, a straightforward calculations gives \eqref{NUSFF30}.   
\end{proof}
Let $M$ be an increasing function in $\R_{+}$. We assume that $i\R\subset\rho(\mathcal{A}_{0})$ and the following growth on the resolvent
\begin{equation}\label{NUSFF3}
\limsup_{\omega\in\R, |\omega| \rightarrow + \infty}\| M(|\omega|)^{-1}(i\omega I-\mathcal{A}_{0})^{-1}\|_{\mathcal{L}(\mathcal{H}_{0})}<+\infty.
\end{equation}
This means according in particular  to Huang-Pr\"uss \cite{huang,pruss} and Batty and Duyckaerts \cite{Ba-Du:08} (see also Borichev and Tomilov theorem \cite[Theorem 2.4]{borichevtomilov}) respectively that the semigroup $e^{t\mathcal{A}_{0}}$ is exponentially stable if $M(|\omega|)=1$ and polynomially stable if  $M(|\omega|)=|\omega|^\ell$  for some $\ell>0$, namely we have
$$
\|\e^{t\mathcal{A}_{0}}\|_{\mathcal{L}(\mathcal{H}_{0})}\leq C\e^{-\delta t},\quad\forall\,t\geq 0
$$
for some $\delta>0$ when $M(|\omega|)=1$  and 
$$
\|\e^{t\mathcal{A}_{0}}(w^{0},w^{1})\|_{\mathcal{H}_{0}}\leq \frac{C}{(1+t)^{\frac{1}{\ell}}}\|(w^{0},w^{1})\|_{\mathcal{D}(\mathcal{A}_{0})},\quad\forall\,t\geq 0,
$$
for all $(w^{0},w^{1})\in \mathcal{D}(\mathcal{A}_{0})$ when $M(|\omega|)=|\omega|^\ell$. However, when $M(|\omega|)=e^{K_{0}|\omega|}$ for some $K_{0}>0$ imply from Burq~\cite{Burq:98} that the semigroup $e^{t\mathcal{A}_{0}}$ is logarithmically stable, namely we have
$$
\|\e^{t\mathcal{A}_{0}}(w^{0},w^{1})\|_{\mathcal{H}_{0}}\leq \frac{C}{\log^{k}(2+t)}\|(w^{0},w^{1})\|_{\mathcal{D}(\mathcal{A}_{0}^{k})},\quad\forall\,t\geq 0,
$$
for every $(w^{0},w^{1})\in \mathcal{D}(\mathcal{A}_{0}^{k})$ and $k\in\N^{*}$.
\begin{thm} \label{th: estimation resolvante}
We assume that $i\R\subset\rho(\mathcal{A})$ and the condition \eqref{NUSFF3} holds. Let $\eta>0$,  there exists a constant $ C>0$ such that
\begin{equation*}
\|(i\omega I-\mathcal{A})^{-1}\|_{\mathcal{L}(\mathcal{H})}\leq C|\omega|^{1-\alpha}M(|\omega|),\quad \forall\,\omega\geq1.
\end{equation*}
\end{thm}
Since $i\R\subset\rho(\mathcal{A})$, then according to Borichev and Tomilov theorem \cite[Theorem 2.4]{borichevtomilov}, we obtain the following corollary. 
\begin{cor}\label{cor : polynomiallement stable}
We assume  that condition \eqref{NUSFF3} holds with $M(|\omega|)=|\omega|^\ell$ for $\ell\geq0$. Then the semigroup $e^{t\mathcal{A}}$ is polynomially stable, namely there exists a constant $C>0$ such that 
$$
\|\e^{t\mathcal{A}}(u^{0},u^{1},\varphi^{0})\|_{\mathcal{H}}\leq \frac{C}{(1+t)^{\frac{1}{1-\alpha+\ell}}}\|(u^{0},u^{1},\varphi^{0})\|_{\mathcal{D}(\mathcal{A})},\quad\forall\,t\geq 0,
$$
 for every initial data $(u^{0},u^{1},\varphi^{0})\in\mathcal{D}(\mathcal{A})$.
In particular, the energy of the strong solution of \eqref{IFF1}-\eqref{IFF3} satisfy the following estimate
$$
E(t)\leq\frac{C}{(1+t)^{\frac{2}{1-\alpha+\ell}}}\|(u^{0},u^{1},0)\|_{\mathcal{D}(\mathcal{A})}^{2}.
$$
\end{cor}
From Batty and Duyckaerts \cite{Ba-Du:08}, see also Burq~\cite{Burq:98} for similar result, we obtain the following corollary.
\begin{cor}\label{cor : logarithme stable}
 We assume  the condition \eqref{NUSFF3} holds with $M(|\omega|)=e^{K_{0}|\omega|}$ for some $K_{0}>0$. Then the semigroup $e^{t\mathcal{A}}$ is logarithmically  stable,  there exists a constant $C>0$ such that 
$$
\|\e^{t\mathcal{A}}(u^{0},u^{1},\varphi^{0})\|_{\mathcal{H}}\leq \frac{C}{\ln^{2}(1+t)}\|(u^{0},u^{1},\varphi^{0})\|_{\mathcal{D}(\mathcal{A})},\quad\forall\,t\geq 0,
$$
for every initial data $(u^{0},u^{1},\varphi^{0})\in\mathcal{D}(\mathcal{A})$.
In particular, the energy of the strong solution of \eqref{IFF1}-\eqref{IFF3} satisfy the following estimate
$$
E(t)\leq\frac{C}{\ln^{2}(1+t)}\|(u^{0},u^{1},0)\|_{\mathcal{D}(\mathcal{A})}^{2}.
$$
\end{cor}
\begin{proof}
We need just to prove that
\begin{equation}\label{NUSFF4}
\limsup_{\omega\in\R, |\omega| \rightarrow + \infty} | \omega| ^{\alpha-1} M(|\omega|)^{-1}\|(i\omega I-\mathcal{A})^{-1}\|_{\mathcal{L}(\mathcal{H})}<+\infty
\end{equation}
is satisfied. For this purpose, we will use an argument of contradiction. We suppose that \eqref{NUSFF4} is false, then there exist a real sequence $(\omega_{n})$, with $\omega_{n}\longrightarrow+\infty$ and a sequence $(u_{n},v_{n},\varphi_{n})\in\mathcal{D}(\mathcal{A})$, verifying the following condition
\begin{equation}\label{NUSFF5}
\|(u_{n},v_{n},\varphi_{n})\|_{\mathcal{H}}=1
\end{equation}
and
\begin{equation}\label{NUSFF6}
\omega_n^{1-\alpha}M(\omega_n)(i\omega_{n}I-\mathcal{A})\left(\begin{array}{c}
u_{n}
\\
v_{n}
\\
\varphi_{n}
\end{array}\right)=\left(\begin{array}{c}
f_{n}
\\
g_{n}
\\
h_{n}
\end{array}\right)\longrightarrow 0\text{ in }\mathcal{H}.
\end{equation}
Multiplying \eqref{NUSFF6} by $ \left(\begin{array}{c}
f_{n}
\\
g_{n}
\\
h_{n}
\end{array}\right)$ and taking the real part of the inner product, we obtain
\begin{equation}\label{NUSFF10}
\re\left\langle\left(\begin{array}{c}
f_{n}
\\
g_{n}
\\
h_{n}
\end{array}\right),\left(\begin{array}{c}
u_{n}
\\
v_{n}
\\
\varphi_{n}
\end{array}\right)\right\rangle_{\mathcal{H}}= \omega_n^{1-\alpha}M(\omega_n)   
\gamma\int_{\R}(|\xi|^{2}+\eta)\|\varphi_{n}(\xi)\|_{U}^{2}\,\ud\xi\underset{n\to+\infty}{\longrightarrow}0.
\end{equation}
Detailing equation \eqref{NUSFF6}, we get
\begin{eqnarray}
 \omega_n^{1-\alpha}M(\omega_n)   (i\omega_{n}u_{n}-v_{n})=f_{n}\longrightarrow 0\text{ in }H_{\frac{1}{2}},\label{NUSFF7}
\\
 \omega_n^{1-\alpha}M(\omega_n)    \left(i\omega_{n}v_{n}+Au_{n}+\gamma B\int_{\R}p(\xi)\varphi_{n}(\xi)\,\ud\xi\right)=g_{n}\longrightarrow 0\text{ in }H,\label{NUSFF8}
\\
 \omega_n^{1-\alpha}M(\omega_n)   (i\omega_{n}\varphi_{n}+(|\xi|^{2}+\eta)\varphi_{n}-p(\xi)B^{*}v_{n})=h_{n}\longrightarrow 0 \text{ in } V.\label{NUSFF9}
\end{eqnarray}
We draw immediately from \eqref{NUSFF7} that 
\begin{equation}\label{NUSFF13}
\omega_{n}\|u_{n}\|_{H}=O(1).
\end{equation}
Taking the inner product of \eqref{NUSFF8} with $u_{n}$ in $H$ and using \eqref{NUSFF7}, one has
\begin{equation*}
\begin{split}
\|u_n\|_{H_{\frac{1}{2}}}^{2}-\omega_{n}^{2}\|u_n\|_{H}^{2}&=-\gamma\left\langle\int_{\R}p(\xi)\varphi_{n}(\xi)\,\ud\xi,B^{*}u_{n}\right\rangle_{U}
\\
&+\omega_n^{\alpha-1}M(\omega_n)^{-1} \left(\langle g_{n},u_{n}\rangle_{H}+i\omega_{n}\langle f_{n},u_{n}\rangle_{H}\right).
\end{split}
\end{equation*}
Using Cauchy-Schwarz inequality, we obtain
\begin{equation*}
\begin{split}
|\|u_{n}\|_{H_{\frac{1}{2}}}^{2}-\omega_{n}^{2}\|u_{n}\|_{H}^{2}|\leq  \omega_n^{\alpha-1}M(\omega_n)^{-1} \|u_{n}\|_{H}(|\omega_{n}|.\|f_{n}\|_{H}+\|g_{n}\|_{H})
\\
+\gamma\|B^{*}u_{n}\|_{U}\left(\int_{\R}\frac{p(\xi)^{2}}{|\xi|^{2}+\eta}\,\ud\xi\right)^{\frac{1}{2}}\left(\int_{\R}(|\xi|^{2}+\eta)\|\varphi_{n}(\xi)\|_{U}^{2}\,\ud\xi\right)^{\frac{1}{2}}.
\end{split}
\end{equation*}
Then \eqref{NUSFF5}-\eqref{NUSFF10} and \eqref{NUSFF13} leads to
\begin{equation}\label{NUSFF15}
\|u_{n}\|_{H_{\frac{1}{2}}}-\omega_{n}\|u_{n}\|_{H}\underset{n\to+\infty}{\longrightarrow}0.
\end{equation}

Following to \eqref{NUSFF7} equations \eqref{NUSFF8} and \eqref{NUSFF9} can be recast as follow
\begin{equation}\label{NUSFF12}
\begin{split}
\varphi_{n}(\xi)&=i\omega_{n}\frac{p(\xi)}{|\xi|^{2}+\eta+i\omega_{n}}B^{*}u_{n}-\omega_n^{\alpha-1}M(\omega_n)^{-1}\frac{p(\xi)}{|\xi|^{2}+\eta+i\omega_{n}}B^{*}f_{n}
\\
&+\omega_n^{\alpha-1}M(\omega_n)^{-1}\frac{h_{n}(\xi)}{|\xi|^{2}+\eta+i\omega_{n}}.
\end{split}
\end{equation}
and
\begin{equation}\label{NUSFF11}
\begin{split}  
-\omega_{n}^{2}u_{n}+Au_{n}+i\omega_{n}\gamma\int_{\R}\frac{p(\xi)^{2}}{|\xi|^{2}+\eta+i\omega_{n}}\,\ud\xi BB^{*}u_{n}
= \omega_n^{\alpha-1}M(\omega_n)^{-1}  (g_{n}+ i\omega_{n}f_{n})
\\
+\gamma
\omega_n^{\alpha-1}M(\omega_n)^{-1} 
\int_{\R}\frac{p(\xi)^{2}}{|\xi|^{2}+\eta+i\omega_{n}}\,\ud\xi BB^{*}f_{n}
-  \omega_n^{\alpha-1}M(\omega_n)^{-1} \gamma
\int_{\R}\frac{p(\xi)Bh_{n}(\xi)}{|\xi|^{2}+\eta+i\omega_{n}}\,\ud\xi
\end{split}
\end{equation}
Multiplying \eqref{NUSFF12} by $|\xi|^{(2-d)/2}$ then integrating over $\R$ with respect to the $\xi$ variable and using Cauchy-Schwarz inequality, we obtain
\begin{align*}
\omega_{n} \left|\int_{\R}\frac{  |\xi|^{ \alpha+1-d}}{|\xi|^{2}+\eta+i\omega_{n}}\,\ud\xi\right|\|B^{*}u_{n}\|_{U}
& \leq  
 \omega_n^{\alpha-1}M(\omega_n)^{-1} 
\left(\int_{\R}\frac{ |\xi|^{ 2-d}  }{(|\xi|^{2}+\eta)^{2}+
\omega_{n}^{2}}\,\ud\xi\right)^{\frac{1}{2}}\|h_{n}\|_{V}
\\  
&\quad +  \omega_n^{\alpha-1}M(\omega_n)^{-1} 
\left|\int_{\R}\frac{    |\xi|^{ \alpha+1-d  }}{|\xi|^{2}+\eta+i\omega_{n}}\,\ud\xi\right|\|B^{*}f_{n}\|_{U}
\\  
& \quad +\left(\int_{\R}\frac{    |\xi|^{2-d}   }{|\xi|^{2}+\eta}\,\ud\xi\right)^{\frac{1}{2}}\left(\int_{\R}(|\xi|^{2}+\eta)\|\varphi_{n}(\xi)\|_{U}^{2}\,\ud\xi\right)^{\frac{1}{2}}.
\end{align*}
Using Lemma \ref{LUSFF8} and Lemma \ref{NUSFF1} we follow
\begin{equation*}
\begin{split}
\omega_{n} (\omega_{n}+\eta)^{(\alpha-1)/2}\|B^{*}u_{n}\|_{U}
\leq C\Bigg(
  \left(\int_{\R}\frac{   |\xi| ^{2-d}  }{|\xi|^{2}+\eta}\,\ud\xi\right)^{\frac{1}{2}}\left(\int_{\R}(|\xi|^{2}+\eta)\|\varphi_{n}(\xi)\|_{U}^{2}\,\ud\xi\right)^{\frac{1}{2}}
\\
+ \omega_n^{\alpha-1}M(\omega_n)^{-1}  (\omega_{n}+\eta)^{-\frac{1}{2}}\|h_{n}\|_{V}
+ \omega_n^{\alpha-1}M(\omega_n)^{-1}  (\omega_{n}+\eta)^{(\alpha-1)/2}\|B^{*}f_{n}\|_{U}\Bigg),
\end{split}
\end{equation*}
which imply from \eqref{NUSFF10} that
\begin{equation}\label{NUSFF14}
\omega_n^2 M(\omega_n) \|B^{*}u_{n}\|_{U}^2 \underset{n\to+\infty}{\longrightarrow}0.
\end{equation}   

Now we recall that the semigroup generated by the operator $\mathcal{A}_{0}$ is stable (in the sense of condition \rfb{NUSFF3}) in the Hilbert space $\mathcal{H}_{0}$ then there exist a unique couple $(w_{n},z_{n})\in\mathcal{D}(\mathcal{A}_{0})$ such that
\begin{equation}\label{NUSFF16}
\left\{\begin{array}{l}
-\omega_{n}^{2}w_{n}+Aw_{n}+i\omega_{n}BB^{*}w_{n}=u_{n}
\\
z_{n}=i\omega_{n}w_{n}
\end{array}\right.
\end{equation}
satisfying the following estimate
\begin{equation}\label{NUSFF17}
\omega_{n}\|w_{n}\|_{H}+\|w_{n}\|_{H_{\frac{1}{2}}}\leq CM(\omega_{n}) \|u_{n}\|_{H},
\end{equation}
since the resolvent of $\mathcal{A}_{0}$ satisfies condition \rfb{NUSFF3}. Next, we take the inner product in $H$ of the first line of \eqref{NUSFF16} with $\omega_{n}w_{n}$, one gets
\begin{equation}\label{NUSFF18}
-\omega_{n}\|\omega_{n}w_{n}\|_{H}^{2}+\omega_{n}\|w_{n}\|_{H_{\frac{1}{2}}}^{2}+i\omega_{n}^{2}\|B^{*}w_{n}\|_{U}^{2}=\omega_{n}\langle u_{n},w_{n}\rangle_{H}.
\end{equation}
Taking the imaginary part of \eqref{NUSFF18}, using Cauchy-Schwarz inequality, and \eqref{NUSFF17}, one gets 
\begin{equation}\label{NUSFF27}
\omega_{n}^{2}\|B^{*}w_{n}\|_{U}^{2}\leq \omega_{n} \| u_n \|_H \|w_n \|_H
 \leq C M(\omega_n ) \| u_n \|_H^2 
\end{equation}
Taking the inner product of \eqref{NUSFF11} with $\omega_{n}^{2}w_{n}$ in the Hilbert space $H$, we have
\begin{align}\label{NUSFF20}
\omega_{n}^{2}\|u_{n}\|_{H}^{2}&=
 \omega_n^{\alpha+1}M(\omega_n)^{-1}  \langle g_{n},w_{n}\rangle_{H}
+i \omega_n^{\alpha+2}M(\omega_n)^{-1}  \langle f_{n},w_{n}\rangle_{H}
\\
&\quad -i\omega_{n}^{3}\gamma\int_{\R}\frac{p(\xi)^{2}}{|\xi|^{2}+\eta+i\omega_{n}}\,\ud\xi\langle B^{*}u_{n},B^{*}w_{n}\rangle_{U}+i\omega_{n}^{3}\langle B^{*}u_{n},B^{*}w_{n}\rangle_{U}     \notag
\\
&\quad
+\omega_n^{\alpha+1}M(\omega_n)^{-1} 
\gamma\int_{\R}\frac{p(\xi)^{2}}{|\xi|^{2}+\eta
+i\omega_{n}}\,\ud\xi
\langle B^{*}f_{n},B^{*}w_{n}\rangle_{U}   \notag
\\
&\quad
-  \omega_n^{\alpha+1}M(\omega_n)^{-1}  \gamma   \int_{\R}p(\xi)\frac{\langle h_{n}(\xi),B^{*}w_{n}\rangle_{U}}{|\xi|^{2}
+\eta+i\omega_{n}}\,\ud\xi.  \notag
\end{align}
Using Lemma \ref{LUSFF8}, Lemma \ref{NUSFF1} and estimates \eqref{NUSFF5}, \eqref{NUSFF6}, \eqref{NUSFF13}, \eqref{NUSFF14}, \eqref{NUSFF17} and \eqref{NUSFF27}, we obtain
\begin{align}\label{NUSFF31}
|\omega_{n}^{3}\langle B^{*}u_{n},B^{*}w_{n}\rangle_{U}|
&\leq \omega_{n}^{3}\|B^{*}u_{n}\|_{U}  \|B^*w_{n}\|_{U}
\leq C\omega_{n}^{2}   M(\omega_n)^{1/2}\|B^{*}u_{n}\|_{U}  \|u_{n}\|_{H}
\\
&
\leq C\omega_{n}    M(\omega_n)^{1/2}\|B^{*}u_{n}\|_{U}  . \omega_n\|u_{n}\|_{H}
\underset{n\to+\infty}{\longrightarrow}0,  \notag
\end{align}
\begin{align}\label{NUSFF21}
\left|\omega_{n}^{3}\int_{\R}\frac{p(\xi)^{2}}{|\xi|^{2}+\eta+i\omega_{n}}\,\ud\xi\langle B^{*}u_{n},B^{*}w_{n}\rangle_{U}\right|
&\leq C\omega_{n}^{2+\alpha} \|B^{*}u_{n}\|_{U}\|B^*w_{n}\|_{U}   
\\
&\leq C\omega_{n}^{\alpha}   M(\omega_n)^{1/2}  \|B^{*}u_{n}\|_{U}.\omega_{n}\|u_{n}\|_{H}
%
%
\underset{n\to+\infty}{\longrightarrow}0,  \notag
\end{align}
\begin{align}\label{NUSFF22}
& \omega_n^{\alpha+1}M(\omega_n)^{-1} \left|  \int_{\R}\frac{p(\xi)^{2}}{|\xi|^{2}+\eta+i\omega_{n}}\,\ud\xi 
\langle B^{*}f_{n},B^{*}w_{n}\rangle_{U}\right| \\
&\qquad    \qquad \qquad \qquad  \qquad \qquad 
\leq C \omega_n^{2\alpha-2}M(\omega_n)^{-1/2}   \|f_{n}\|_{H_{\frac{1}{2}}}.
\omega_{n}\|u_{n}\|_{H}\underset{n\to+\infty}{\longrightarrow}0,\notag
\end{align}  
\begin{align}\label{NUSFF23}
 \omega_n^{\alpha+1}M(\omega_n)^{-1} &  \left|   \int_{\R}p(\xi)\frac{\langle h_{n}(\xi),B^{*}w_{n}\rangle_{U}}{|\xi|^{2}
+\eta+i\omega_{n}}\,\ud\xi \right| 
\\
&\leq   
 \omega_n^{\alpha-1}M(\omega_n)^{-1}    \left(\int_{\R}\frac{p(\xi)^{2}}{(|\xi|^{2}+\eta)^{2}+\omega_{n}^{2}}\,
\ud\xi\right)^{\frac{1}{2}}\|h_{n}\|_{V}.\omega_{n}\|u_{n}\|_{H}   \notag
\\
&\leq C\omega_n^{\alpha-1 } \omega_{n}^{\alpha/2-1}  M(\omega_n)^{-1}     \|h_{n}\|_{V}.\omega_{n}\|u_{n}\|_{H}\underset{n\to+\infty}{\longrightarrow}0,
\notag
\end{align}      
and
\begin{equation}\label{NUSFF24}
 \omega_n^{\alpha+1}M(\omega_n)^{-1} |\langle g_{n},w_{n}\rangle_{H}|
\leq C\omega_n^{\alpha-1} \|g_{n}\|_{H}.\omega_{n}\|u_{n}\|_{H}\underset{n\to+\infty}{\longrightarrow}0.
\end{equation}  
Taking the inner product of the first equation of \eqref{NUSFF16} with $f_{n}$, we obtain
\begin{equation*}
-\omega_{n}^{2}\langle w_{n},f_{n}\rangle_{H}+\langle A^{\frac{1}{2}}w_{n},A^{\frac{1}{2}}f_{n}\rangle_{H}+i\omega_{n}\langle B^{*}w_{n},B^{*}f_{n}\rangle_{U}=\langle u_{n},f_{n}\rangle_{H}.
\end{equation*}
This with \eqref{NUSFF5}, \eqref{NUSFF6}, \eqref{NUSFF13}, \eqref{NUSFF17} and \eqref{NUSFF27} give
\begin{align}\label{NUSFF25}
 \omega_n^{\alpha+2}&M(\omega_n)^{-1} |\langle w_{n},f_{n}\rangle_{H}|
 \\
&\leq    \omega_n^{\alpha}M(\omega_n)^{-1}     (\|w_{n}\|_{H_{\frac{1}{2}}}\|f_{n}\|_{H_{\frac{1}{2}}}+\|u_{n}\|_{H}\|f_{n}\|_{H})  \notag
%
+\omega_n^{\alpha+1}M(\omega_n)^{-1}    \|B^{*}w_{n}\|_{U}\|B^{*}f_{n}\|_{U}
\\
&\leq C  \omega_n^{\alpha-1}  (1+ M ( \omega_{n}) ^{-1}  +M ( \omega_{n}) ^{-\frac{1}{2}}) \|f_{n}\|_{H_{\frac{1}{2}}}.\omega_{n}\|u_{n}\|_{H}\underset{n\to+\infty}{\longrightarrow}0.   \notag
\end{align}    
It follows from the combination of \eqref{NUSFF20} and \eqref{NUSFF31}-\eqref{NUSFF25} that $\|\omega_{n}u_{n}\|_{H}\underset{n\to+\infty}{\longrightarrow}0$. Thus, by \eqref{NUSFF15} we have $\|u_{n}\|_{H_{\frac{1}{2}}}\underset{n\to+\infty}{\longrightarrow}0$. Together with \eqref{NUSFF7} and \eqref{NUSFF10} imply that $(u_{n},v_{n},\varphi_{n})\underset{n\to+\infty}{\longrightarrow}0$ which contradicts \eqref{NUSFF5}. This completes the proof.
\end{proof}

\begin{rem}

In the case where for all $\delta >0, \ds \sup_{Re \lambda = \delta} \left\|\lambda B^* (\lambda^2 I + A)^{-1} B\right\|_{\mathcal{L}(U)} <\infty,$ according to \cite{AT} (see also \cite{ammariniciase}), we can replace the hypothesis \rfb{NUSFF3} by the following observability inequalities and we obtain the same results:

\begin{itemize}
\item
for $\ell =0$, the assumption \rfb{NUSFF3} is equivalent to the following exact observability inequality: there exists $T, C > 0$ such that 
$$
\int_0^T \left\|(0 \,\, B^{*}) e^{t \left(\begin{array}{rc}
0&I
\\
-A&0
\end{array}\right)} \left( \begin{array}{l}
u^{0} \\
u^{1}
\end{array}
\right) \right\|^2_U \, \ud t \geq C \, \left\|(u^0,u^1)\right\|^2_{\mathcal {H}_0}, \, \forall \, (u^{0},u^{1}) \in H_{1} \times H,
$$
and
\item 
for $\ell > 0$, the assumption \rfb{NUSFF3} is an implication for the following weak observability inequality: there exists $T,\,C > 0$ such that 

$$
\int_0^T \left\|(0 \,\, B^{*}) e^{t \left(\begin{array}{rc}
0&I
\\
-A&0
\end{array}\right)} \left( \begin{array}{l}
u^{0} \\
u^{1}
\end{array}
\right) \right\|^2_U \, \ud t \geq C \, \left\|(u^0,u^1)\right\|^2_{H_{- \frac{- \alpha +\ell}{2}} \times H_{- \frac{1 - \alpha+\ell}{2}}}, \, \forall \, (u^{0},u^{1}) \in H_{1}\times H.
$$
\end{itemize}
\end{rem}
\section{Applications to the fractional-damped wave equation}\label{AFF}
\subsection{Internal fractional-damped wave equation}
We consider a wave equation with an internal fractional-damping in a bounded and connected domain $\Omega$ of $\R^{n}$ with smooth boundary $\Gamma=\partial\Omega$
\begin{equation}\label{AFF1}
\left\{\begin{array}{ll}
\partial_{t}^{2}u(x,t)-\Delta u(x,t)+a(x)\partial^{\alpha,\eta}_{t}u(x,t)=0&\text{in }\Omega\times \R_{+}
\\
u(x,t)=0&\text{on }\Gamma\times \R_{+}
\\
u(x,0)=u^{0}(x),\quad\partial_{t}u(x,0)=u^{1}(x)&\text{in }\Omega,
\end{array}\right.
\end{equation}
where $a(x)$ is a positive function in $\Omega$ with support $\omega_{0}=\mathrm{supp}(a)$ verifying that there exist a non empty open subset $\widetilde{\omega}_{0}\subset\Omega$ and a strictly positive constant $a_{0}$ such that
$$
a(x)\geq a_{0} \qquad\forall\,x\in\widetilde{\omega}_{0}.
$$
System \eqref{AFF1} can be recast as follow
\begin{equation}\label{AFF2}
\left\{\begin{array}{ll}
\ds\partial_{t}^{2}u(x,t)-\Delta u(x,t)+\gamma\sqrt{a(x)}\int_{\R}p(\xi)\varphi(x,t,\xi)\,\ud\xi=0&(x,t)\in\Omega\times \R_{+}
\\
\partial_{t}\varphi(x,t,\xi)+(|\xi|^2+\eta)\,\varphi(x,t,\xi)=p(\xi)\sqrt{a(x)}\partial_{t}u(x,t)&(x,t,\xi)\in\Omega\times\R_{+}\times\R
\\
u(x,t)=0&(x,t)\in\Gamma\times\R_{+}
\\
u(x,0)=u^{0}(x),\quad\partial_{t}u(x,0)=u^{1}(x),\quad\varphi(x,0,\xi)=0&x\in\Omega,\;\xi\in\R.
\end{array}\right.
\end{equation}
The energy of the system is given by
$$
E(t)=\frac{1}{2}\left(\|\partial_{t}u(t)\|_{L^{2}(\Omega)}^{2}+\|\nabla u(t)\|_{L^{2}(\Omega)}^{2}+\gamma\int_{\R}\|\varphi(t,\xi)\|_{L^{2}(\Omega)}^{2}\,\ud\xi\right).
$$
The operator $A=-\Delta$ is strictly positive and self-adjoint operator in $H=L^{2}(\Omega)$ and with domain $\mathcal{D}(A)=H_{0}^{1}(\Omega)\cap H^{2}(\Omega)$. The operator $\mathcal{A}$ corresponding to the Cauchy problem of system \eqref{AFF2} is given by
$$
\mathcal{A}\left(\begin{array}{c}
u
\\
v
\\
\varphi
\end{array}\right)=\left(\begin{array}{c}
v
\\
\ds\Delta u-\gamma\sqrt{a}\int_{\R}p(\xi)\varphi(\xi)\,\ud\xi
\\
-(|\xi|^2+\eta)\varphi(\xi)+p(\xi)\sqrt{a}v
\end{array}\right)
$$
with domain in the Hilbert space $\mathcal{H}=H_{0}^{1}(\Omega)\times L^{2}(\Omega)\times L^{2}(\R;L^{2}(\Omega))$ given by
\begin{equation*}
\begin{split}
\mathcal{D}(\mathcal{A})=\Big\{(u,v,\varphi)\in\mathcal{H}:\,v\in H_{0}^{1}(\Omega),\,\ds\Delta u-\gamma\sqrt{a}\int_{\R}p(\xi)\varphi(\xi)\,\ud\xi\in L^{2}(\Omega),
\\
|\xi|\varphi\in L^{2}(\R;L^{2}(\Omega)),\,(|\xi|^2+\eta)\varphi(\xi)-p(\xi)\sqrt{a}v\in L^{2}(\R;L^{2}(\Omega))\Big\}.
\end{split}
\end{equation*}
Since the embedding $H_{0}^{1}(\Omega)\hookrightarrow L^{2}(\Omega)$ is compact and the only solution of the following problem
\begin{equation*}
\left\{\begin{array}{ll}
\partial_{t}^{2}u(x,t)-\Delta u(x,t)=0&(x,t)\in\Omega\times\R_{+}
\\
\sqrt{a(x)}\partial_{t}u(x,t)=0&(x,t)\in\Omega\times\R_{+}
\\
u(x,t)=0&(x,t)\in\Gamma\times\R_{+},
\end{array}\right.
\end{equation*}
is the trivial solution (see proof of Proposition \ref{AFF19}), then according to section \ref{SSFF} the semigroup $e^{t\mathcal{A}}$ is strongly stable. Moreover, we have the following lemma (for proof look at those of Lemma \ref{AFF9} and Lemma \ref{AFF21}).
\begin{lem}\label{AFF10}
Let $\eta>0$ and for all $\omega\in\R$ the operator $(i\omega I-\mathcal{A})$ is injective and surjective.
\end{lem}
We assume that the semigroup of the operator $\mathcal{A}_{0}:\mathcal{D}(\mathcal{A})\subset\mathcal{H}_{0}\longrightarrow\mathcal{H}_{0}$ defined by
$$
\mathcal{A}_{0}\left(\begin{array}{c}
u
\\
v
\end{array}\right)=\left(\begin{array}{c}
v
\\
\Delta u-av
\end{array}\right)
$$
where $\mathcal{H}_{0}=H_{0}^{1}(\Omega)\times L^{2}(\Omega)$ with domain
$$
\mathcal{D}(\mathcal{A}_{0})=\{(u,v)\in\mathcal{H}_{0}:\; \Delta u-av\in L^{2}(\Omega),\;v\in H_{0}^{1}(\Omega)\},
$$
is uniformly stable in the energy space $\mathcal{H}_{0}$, which means that the energy of the following system
\begin{equation*}
\left\{\begin{array}{ll}
\partial_{t}^{2}w(x,t)-\Delta w(x,t)+ a(x)\partial_{t}w(x,t)=0&\text{in }\Omega\times \R_{+}
\\
w(x,t)=0&\text{on }\Gamma\times \R_{+}
\\
w(x,0)=w^{0}(x),\quad\partial_{t}w(x,0)=w^{1}(x)&\text{in }\Omega.
\end{array}\right.
\end{equation*}
is exponentially stable. Noting that this can be held if the so called geometric control condition (GCC) is satisfied (see \cite{blr}).
\begin{prop}
Under the above assumption and for $\eta>0$ the operator $\mathcal{A}$ generates a  contraction semigroup satisfying
$$
\|e^{t\mathcal{A}}X\|_{\mathcal{H}}\leq\frac{C}{(1+t)^{\frac{1}{1-\alpha}}}\|X\|_{\mathcal{D}(\mathcal{A})},\quad\forall\,X\in\mathcal{D}(\mathcal{A}),\;t\geq 0,
$$
for some constant $C>0$. This means that the energy of system \eqref{AFF1} is decreasing to zero as $t$ goes to $+\infty$ as $t^{\frac{-2}{1-\alpha}}$.
\end{prop}
\begin{proof}
Following to Lemma \ref{AFF10} the operator $(i\omega I-\mathcal{A})$ is bijective for every $\omega\in\R$, then using the closed graph theorem  we follow that $i\R\subset\rho(\mathcal{A})$. The result follow now from 
Corollary~\ref{cor : polynomiallement stable}.
\end{proof}
\begin{rem}
In the case where $\Omega = (0,1) \times (0,1)$ and
$$
a(x) = \left\{
\begin{array}{ll}
1, \, \forall \, x \in (0,\varepsilon) \times (0,1),\\
0, \, \text{elsewhere},
\end{array}
\right.,
$$
where $\varepsilon>0$ is a constant, we have according to \cite{stahn} that the semigroup generated by the operator $\mathcal{A}_{0}$ decays as $t^{-\frac{3}{2}}$ (which is optimal). We obtain in this case from Corollary \ref{cor : polynomiallement stable} that the polynomial decay rate for the semigroup $e^{t\mathcal{A}}$ is given by $t^{-\frac{2}{5-2 \alpha}}$.

However, we obtain a logarithm decay rate of the semigroup $e^{t\mathcal{A}}$ as given in Corollary \ref{cor : logarithme stable} without any geometrical condition since according to \cite{lebeau} the resolvent of the operator $\mathcal{A}_{0}$ satisfies the condition \eqref{NUSFF3} with $M(|\omega|)=\e^{K_{0}|\omega|}$ for some $K_{0}>0$.
\end{rem}
\subsection{Fractional-Kelvin-Voigt damped wave equation}
We consider the following damped wave system
\begin{equation*}
\left\{\begin {array}{ll}
\partial_{t}^{2}u(x,t)-\Delta u(x,t)-\mathrm{div}\left(a(x)\nabla\partial_{t}^{\alpha,\eta}u(x,t)\right)=0&(x,t)\in\Omega\times \R_{+}
\\
u(x,t)=0&(x,t)\in\Gamma\times\R_{+}
\\
u(x,0)=u^{0}(x),\quad \partial_{t}u(x,0)=u^{1}(x)&x\in\Omega,
\end{array}\right.
\end{equation*}
where we have made the same notations as the previous subsection. Equivalently, we have
\begin{equation}\label{AFF11}
\left\{\begin{array}{ll}
\ds\partial_{t}^{2}u(x,t)-\Delta u(x,t)-\gamma\mathrm{div}\left(\sqrt{a(x)}\int_{\R}p(\xi)\varphi(x,t,\xi)\,\ud\xi\right)=0&(x,t)\in\Omega\times \R_{+}
\\
\partial_{t}\varphi(x,t,\xi)+(|\xi|^2+\eta)\,\varphi(x,t,\xi)=p(\xi)\sqrt{a(x)}\nabla\partial_{t}u(x,t)&(x,t,\xi)\in\Omega\times\R_{+}\times\R
\\
u(x,t)=0&(x,t)\in\Gamma\times\R_{+}
\\
u(x,0)=u^{0}(x),\quad\partial_{t}u(x,0)=u^{1}(x),\quad\varphi(x,0,\xi)=0&x\in\Omega,\;\xi\in\R.
\end{array}\right.
\end{equation}
The energy of the system is given by
$$
E(t)=\frac{1}{2}\left(\|\partial_{t}u(t)\|_{L^{2}(\Omega)}^{2}+\|\nabla u(t)\|_{L^{2}(\Omega)}^{2}+\gamma\int_{\R}\|\varphi(t,\xi)\|_{(L^{2}(\Omega))^{n}}^{2}\,\ud\xi\right).
$$
The operator $A=-\Delta$ is strictly positive and auto-adjoint operator in $H=L^{2}(\Omega)$ and with domain $\mathcal{D}(A)=H_{0}^{1}(\Omega)\cap H^{2}(\Omega)$. The operator $\mathcal{A}$ corresponding to the Cauchy problem of system \eqref{AFF11} is given by
$$
\mathcal{A}\left(\begin{array}{c}
u
\\
v
\\
\varphi
\end{array}\right)=\left(\begin{array}{c}
v
\\
\ds\Delta u+\gamma\mathrm{div}\left(\sqrt{a}\int_{\R}p(\xi)\varphi(\xi)\,\ud\xi\right)
\\
-(|\xi|^2+\eta)\varphi(\xi)+p(\xi)\sqrt{a}\nabla v
\end{array}\right)
$$
with domain in the Hilbert space $\mathcal{H}=H_{0}^{1}(\Omega)\times L^{2}(\Omega)\times L^{2}(\R;(L^{2}(\Omega))^{n})$ is given by
\begin{equation*}
\begin{split}
\mathcal{D}(\mathcal{A})=\Big\{(u,v,\varphi)\in\mathcal{H}:\,v\in H_{0}^{1}(\Omega),\,\ds\Delta u+\gamma\mathrm{div}\left(\sqrt{a}\int_{\R}p(\xi)\varphi(\xi)\,\ud\xi\right)\in L^{2}(\Omega),
\\
|\xi|\varphi\in L^{2}(\R;(L^{2}(\Omega))^{n}),\,(|\xi|^2+\eta)\varphi(\xi)-p(\xi)\sqrt{a}\nabla v\in L^{2}(\R;(L^{2}(\Omega))^{n})\Big\}.
\end{split}
\end{equation*}
\begin{prop}\label{AFF19}
The operator $\mathcal{A}$ generates a $C_{0}$-semigroup of contraction therefore, system \eqref{AFF11} is well posed in the energy space $\mathcal{H}$. Moreover, if we assume that the intersection of the boundary of $\Omega$ with any connected component of $\omega_{0}=\mathrm{supp}(a)$ is nonempty subset with a non-zero measure then the semigroup is strongly stable.
\end{prop}
\begin{proof}
The well-posedness follows from Theorem \ref{WFF1}. Since $H_{0}^{1}(\Omega)\hookrightarrow L^{2}(\Omega)$ is compact embedding, then thanks to Theorem \ref{SSFF20} strong stabilization will be guarantee if we prove that the only solution of the problem 
\begin{equation}\label{AFF14}
\left\{\begin {array}{ll}
\partial_{t}^{2}u(x,t)-\Delta u(x,t)=0&(x,t)\in\Omega\times \R_{+}
\\
\sqrt{a(x)}\nabla\partial_{t}u(x,t)=0&(x,t)\in\Omega\times \R_{+}
\\
u(x,t)=0&(x,t)\in\Gamma\times\R_{+},
\end{array}\right.
\end{equation}
such that $\Delta u(x,t)\in L^{2}(\Omega)$, $\partial_{t}u(x,t)\in H_{0}^{1}(\Omega)$ for all $t\geq 0$, is the zero solution. Without lost of generality we can assume that $\omega_{0}$ is connected. Using the second line of \eqref{AFF14} we can see easily that $\nabla u(x,t)$ it does not depend on the time variable and $\partial_{t}u(x,t)$ is independent of the space variable in $\omega_{0}$, say that $\nabla u(x,t)=f(x)$ and $\partial_{t}u(x,t)=g(t)$ in $\omega_{0}$. Putting these two equalities into the first equation of \eqref{AFF14} then we follow that $\partial_{t}^{2}u$ is constant in $\omega_{0}$. The Combination of all these properties of $u$ imply  that the solution of \eqref{AFF14} is written in $\omega_{0}$ as follow $u(x,t)=\beta t^{2}+\delta t+\phi(x)$, where $\beta$ and $\delta$ are two real numbers. Since, $u\equiv0$ on $\partial\omega_{0}\cap\Gamma$ then $u$ does not depend on the time variable and we get $u=\phi$ in $\omega_{0}$. 
\\
We set now $v(x,t)=\partial_{t}u(x,t)$, then $v$ satisfies the following system
\begin{equation*}
\left\{\begin{array}{ll}
\partial_{t}^{2}v(x,t)-\Delta v(x,t)=0&(x,t)\in\Omega\times\R_{+}
\\
v(x,t)=0&(x,t)\in\omega\times\R_{+}
\\
v(x,t)=0&(x,t)\in\Gamma\times\R_{+}.
\end{array}\right.
\end{equation*}
Since $v(t)\in H_{0}^{1}(\Omega)$ for all $t\geq 0$ then using the unique continuation theorem (see \cite{lions, zuazua}) we find that $v\equiv 0$ in $\Omega$. This means that $u$ is only depends on the $x$ variable and verifying the following system of equations
\begin{equation*}
\left\{\begin{array}{ll}
-\Delta u(x)=0&x\in\Omega
\\
u(x)=0&x\in\Gamma.
\end{array}\right.
\end{equation*}
Since the Dirichlet Laplacian operator is invertible we follow that $u\equiv0$ in $\Omega$. And this completes the proof.
\end{proof}
\begin{lem}\label{AFF9}
For all $\omega\in\R$ the operator $(i\omega I-\mathcal{A})$ is injective.
\end{lem}
\begin{proof}
Let $X=\left(\begin{array}{c}
u
\\
v
\\
\varphi
\end{array}\right)\in\mathcal{D}(\mathcal{A})$ such that
\begin{equation}\label{AFF3}
\mathcal{A}X=i\omega X
\end{equation}
Then the dissipation property of the operator $\mathcal{A}$ imply that
$$
\re\langle\mathcal{A}X,X\rangle=-\gamma\int_{\R}(|\xi|^{2}+\eta)\|\varphi(\xi)\|_{(L^{2}(\Omega))^{n}}^{2}\,\ud\xi=0.
$$
Then we deduce that
$$
\varphi(\xi)=0\quad\text{in }(L^{2}(\Omega))^{n}\text{ a.e }\xi\in\R.
$$
Since that problem \eqref{AFF3} becomes
\begin{equation*}
\left\{\begin{array}{ll}
v=i\omega u&\text{in }\Omega
\\
\omega^{2}u+\Delta u=0&\text{in }\Omega
\\
\sqrt{a(x)}\nabla u=0&\text{in }\mathrm{supp}(a)
\\
u=0&\text{on }\Gamma.
\end{array}\right.
\end{equation*}
We denote by $w_{j}=\partial_{x_{j}}u$ and we derive the second and the third equation, one gets
\begin{equation*}
\left\{\begin{array}{ll}
\omega^{2}w_{j}+\Delta w_{j}=0&\text{in }\Omega
\\
w_{j}=0&\text{in }\mathrm{supp}(a),
\end{array}\right.
\end{equation*}
By unique continuation theorem we find that $w_{j}=0$ in $\Omega$ therefore $u=0$ in $\Omega$ since $u_{|\Gamma}=0$ and consequently $U=0$. Thus, the injection of the operator $(i\omega I-\mathcal{A})$ is proven.
\end{proof}
\begin{lem}\label{AFF7}
Assume that $\eta>0$ and $\omega\in\R$ then for any $f\in H^{-1}(\Omega)$ the following problem
\begin{equation}\label{AFF4}
\left\{\begin{array}{ll}
\omega^{2}u+\Delta u+(\omega^{2}c_{1}+i\omega c_{2})\mathrm{div}(a\nabla u)=f&\text{in }\Omega
\\
u=0&\text{on }\Gamma
\end{array}\right.
\end{equation}
where
$$
c_{1}=\gamma\int_{\R}\frac{p(\xi)^{2}}{(|\xi|^{2}+\eta)^{2}+\omega^{2}}\,\ud\xi\qquad\text{and}\qquad c_{2}=\gamma\int_{\R}\frac{p(\xi)^{2}(|\xi|^{2}+\eta)}{(|\xi|^{2}+\eta)^{2}+\omega^{2}}\,\ud\xi
$$
admits a unique solution $u\in H_{0}^{1}(\Omega)$.
\end{lem}
\begin{proof}
First we note that the coefficients $c_{1}$ and $c_{2}$ are well defined. We distinguish two cases:

\underline{\textbf{Case 1:} $\omega=0$.} In this case we use the Lax-Milgram's theorem to prove the unique solution $u\in H_{0}^{1}(\Omega)$ of \eqref{AFF4}.

\underline{\textbf{Case 2:} $\omega\in\R^{*}$.} Separating the real and the imaginary parts of $u$ and $f$ by writing $u=u_{1}+iu_{2}$ and $f=f_{1}+if_{2}$ and we consider the following mixed system
\begin{equation}\label{AFF5}
\left\{\begin{array}{ll}
-\Delta u_{1}-c_{1}\omega^{2}\mathrm{div}(a\nabla u_{1})+c_{2}\omega\mathrm{div}(a\nabla u_{2})=f_{1}&\text{in }\Omega
\\
-\Delta u_{2}-c_{1}\omega^{2}\mathrm{div}(a\nabla u_{2})-c_{2}\omega\mathrm{div}(a\nabla u_{1})=f_{2}&\text{in }\Omega
\\
u_{1}=u_{2}=0&\text{on }\Gamma.
\end{array}\right.
\end{equation}
Consider the following bilinear form in $(H_{0}^{1}(\Omega)\times H_{0}^{1}(\Omega))^{2}$ defined by
\begin{equation*}
\begin{split}
L((u_{1},u_{2});(w_{1},w_{2}))&=\int_{\Omega}\nabla u_{1}\nabla w_{1}\,\ud x+\int_{\Omega}\nabla u_{2}\nabla w_{2}\,\ud x+c_{1}\omega^{2}\int_{\Omega}\nabla u_{1}\nabla w_{1}a\,\ud x
\\
&+c_{1}\omega^{2}\int_{\Omega}\nabla u_{2}\nabla w_{2}a\,\ud x-c_{2}\omega\int_{\Omega}\nabla u_{2}\nabla w_{1}a\,\ud x+c_{2}\omega\int_{\Omega}\nabla u_{1}\nabla w_{2}a\,\ud x.
\end{split}
\end{equation*}
It is clear that $L$ is continuous and coercive in $(H_{0}^{1}(\Omega)\times H_{0}^{1}(\Omega))^{2}$ then by Lax-Milgram's theorem there exists a unique couple $(u_{1},u_{2})\in H_{0}^{1}(\Omega)\times H_{0}^{1}(\Omega)$ such that
$$
L((u_{1},u_{2});(w_{1},w_{2}))=\langle f_{1},w_{1}\rangle_{H^{-1}\times H_{0}^{1}}+\langle f_{2},w_{2}\rangle_{H^{-1}\times H_{0}^{1}},\quad \forall\,(w_{1},w_{2})\in H_{0}^{1}(\Omega)\times H_{0}^{1}(\Omega).
$$
This leads to the existence and the uniqueness of a solution of the problem \eqref{AFF5} in $H_{0}^{1}(\Omega)\times H_{0}^{1}(\Omega)$.

This proves in particular that the operator $A_{\omega}=-\Delta-(\omega^{2}c_{1}+i\omega c_{2})\mathrm{div}(a\nabla\,.\,)$ is invertible from $H_{0}^{1}(\Omega)$ into $H^{-1}(\Omega)$ then the first line of \eqref{AFF4} is equivalent to the following equation
\begin{equation}\label{AFF22}
(\omega^{2}A_{\omega}^{-1}-I)u=A_{\omega}^{-1}f.
\end{equation}
It follows from the compactness of the embedding $H_{0}^{1}(\Omega)\hookrightarrow H^{-1}(\Omega)$ that the inverse operator $A_{\omega}^{-1}$ is compact in $H^{-1}(\Omega)$. Let's consider the following problem 
\begin{equation}\label{AFF12}
\left\{\begin{array}{ll}
\omega^{2}u+\Delta u+(\omega^{2}c_{1}+i\omega c_{2})\mathrm{div}(a\nabla u)=0&\text{in }\Omega
\\
u=0&\text{in }\Gamma,
\end{array}\right.
\end{equation}
we multiplying the first line of \eqref{AFF12} by $\overline{u}$ and integrating over $\Omega$, one gets
\begin{equation}\label{AFF13}
\omega^{2}\|u\|_{L^{2}(\Omega)}^{2}-\|\nabla u\|_{L^{2}(\Omega)}^{2}-(\omega^{2}c_{1}+i\omega c_{2})\|\sqrt{a}\nabla u\|_{L^{2}(\Omega)}^{2}=0,
\end{equation}
then by taking the imaginary part of \eqref{AFF13} we obtain $\nabla u=0$ in $\mathrm{supp}(a)$. Proceeding as the proof of the previous lemma one gets $u=0$ in $\Omega$. This prove that the operator $(\omega^{2}A_{\omega}^{-1}-I)$ is injective. Then following to Fredhom's alternative theorem \cite[Th\'eoreme 6.6]{brezis}, equation \eqref{AFF22} admits a unique solution and therefore equation \eqref{AFF4} admits a unique solution.
\end{proof}
\begin{lem}\label{AFF21}
Let $\eta>0$ and $a$ smooth enough then for all $\omega\in\R$ the operator $(i\omega I-\mathcal{A})$ is surjective.
\end{lem}
\begin{proof}
Let $Y=(f,g,h)\in\mathcal{H}$ and we look for an $X=(u,v,\varphi)\in\mathcal{D}(\mathcal{A})$ such that
\begin{equation}\label{AFF6}
(i\omega I-\mathcal{A})X=Y.
\end{equation}
Equivalently, we have
\begin{equation}\label{AFF8}
\left\{\begin{array}{ll}
v=i\omega u-f&\text{in }\Omega
\\
\omega^{2}u+\Delta u+(\omega^{2}c_{1}+i\omega c_{2})\mathrm{div}(a\nabla u)=F&\text{in }\Omega
\\
\ds\varphi(\xi)=i\omega\frac{p(\xi)}{|\xi|^{2}+\eta+i\omega}\sqrt{a}\nabla u-\frac{p(\xi)}{|\xi|^{2}+\eta+i\omega}\sqrt{a}\nabla f+\frac{h(\xi)}{|\xi|^{2}+\eta+i\omega}&\text{in }\Omega
\\
u=0&\text{on }\Gamma,
\end{array}\right.
\end{equation}
where $c_{1}$ and $c_{2}$ are defined in Lemma \ref{AFF7} and $F\in L^{2}(\Omega)$ is given by
$$
F=(c_{2}-i\omega c_{1})\mathrm{div}(a\nabla f)-i\omega f-g-\gamma\mathrm{div}\left(\sqrt{a}\int_{\R}p(\xi)\frac{h(x,\xi)}{|\xi|^{2}+\eta+i\omega}\,\ud\xi\right).
$$
Since for $a$ smooth enough $F\in H^{-1}(\Omega)$ then using Lemma \ref{AFF7}, problem \eqref{AFF8} has a unique solution $u\in H_{0}^{1}(\Omega)$ and therefore problem \eqref{AFF6} has a unique solution $X\in\mathcal{D}(\mathcal{A})$.
\end{proof}

We consider now the following auxiliary problem
\begin{equation}\label{AFF16}
\left\{\begin{array}{ll}
\partial_{t}^{2}w(x,t)-\Delta w(x,t)+\mathrm{div}(a(x)\nabla\partial_{t}w(x,t))=0&\text{in }\Omega\times \R_{+}
\\
w(x,t)=0&\text{on }\Gamma\times \R_{+}
\\
w(x,0)=w^{0}(x),\quad\partial_{t}w(x,0)=w^{1}(x)&\text{in }\Omega.
\end{array}\right.
\end{equation}
The equation \eqref{AFF16} is well posed in the Hilbert space $\mathcal{H}_{0}=H_{0}^{1}(\Omega)\times L^{2}(\Omega)$ and its solution is a semigroup generated by the operator $\mathcal{A}_{0}:\mathcal{D}(\mathcal{A})\subset\mathcal{H}_{0}\longrightarrow\mathcal{H}_{0}$ defined by
$$
\mathcal{A}_{0}\left(\begin{array}{c}
w
\\
v
\end{array}\right)=\left(\begin{array}{c}
v
\\
\Delta w-\mathrm{div}(a\nabla v)
\end{array}\right)
$$
with domain
$$
\mathcal{D}(\mathcal{A}_{0})=\{(w,v)\in\mathcal{H}_{0}:\; \Delta w-\mathrm{div}(a\nabla v)\in L^{2}(\Omega),\;v\in H_{0}^{1}(\Omega)\}.
$$
We suppose now that there exist $\Omega_{j}\subset\Omega$ with piecewise smooth boundary and $x_{0}^{j}\in\R^{n}$, $j=1,2,\ldots,J$ s.t. $\Omega_{i}\cap\Omega_{j}=\emptyset$ for any $1\leq i<j\leq J$ and for some $\delta>0$,
$$
\Omega\cap\mathcal{N}_{\delta}\left[\left(\bigcup_{j=1}^{J}\Gamma_{j}\right)\bigcup\left(\Omega\setminus\bigcup_{j=1}^{J}\Omega_{j}\right)\right]\subset\omega_{0},
$$
where for $\ds S\subset\R^{n}$, $\mathcal{N}_{\delta}(S)=\bigcup_{x\in S}\left\{y\in\R^{n};|x-y|<\delta\right\}$ and $\Gamma_{j}=\left\{x\in\partial\Omega_{j};\, (x-x_{0}^{j}).\nu^{j}(x)>0\right\}$, $\nu^{j}$ being the unit normal vector pointing into the exterior of $\Omega_{j}$. Under the above assumptions Tebou in \cite[Theorem 1.1]{tibou1} shows a non-uniform stabilization result in such a way the energy of the system \eqref{AFF16} decreases to zero as $t^{-1}$ as $t$ goes to the infinity for regular initial data. Consequently, according to Corollary \ref{cor : polynomiallement stable} we have the following 
\begin{prop}
Under the above assumptions and for $\eta>0$ the operator $\mathcal{A}$ generates a  C$_0$ semigroup of contractions satisfying
$$
\|e^{t\mathcal{A}}X\|_{\mathcal{H}}\leq\frac{C}{(1+t)^{\frac{1}{2-\alpha}}}\|X\|_{\mathcal{D}(\mathcal{A})},\quad\forall\,X\in\mathcal{D}(\mathcal{A}),\;t\geq 0,
$$
for some constant $C>0$. This means that the energy of system \eqref{AFF11} is decreasing to zero as $t$ goes to $+\infty$ as $t^{-\frac{2}{2-\alpha}}$.
\end{prop}
If in addition to the above geometric condition above we have that we have the following regularity of the damping coefficient $a\in W^{1,\infty}(\Omega)$, $|\nabla a(x)|^{2}M_{0}a(x)$ a.e. in $\Omega$, $a(x)\geq a_{0}$ a.e. in $\Omega$, Tebou in \cite[Theorem 1.2]{tibou1} (see also \cite[Theorem 1.2]{tibou2}) shows that the semigroup generated by the operator $\mathcal{A}_{0}$ is uniformly stable that is the energy of the system \eqref{AFF16} is exponentially stable. Then by combining this with Corollary \ref{cor : polynomiallement stable} we get the following 
\begin{prop}
Under the above additional assumptions and for $\eta>0$ the operator $\mathcal{A}$ generates a  C$_0$ semigroup of contractions satisfying
$$
\|e^{t\mathcal{A}}X\|_{\mathcal{H}}\leq\frac{C}{(1+t)^{\frac{1}{1-\alpha}}}\|X\|_{\mathcal{D}(\mathcal{A})},\quad\forall\,X\in\mathcal{D}(\mathcal{A}),\;t\geq 0,
$$
for some constant $C>0$. This means that the energy of system \eqref{AFF11} is decreasing to zero as $t$ goes to $+\infty$ as $t^{-\frac{2}{1-\alpha}}$.
\end{prop}
However without any geometric conditions namely $a$ equal to a constant $d$ in $\omega_{0}$ and equal to zero elsewhere, using Carleman estimate Ammari, Hassine and Robbinao \cite[Theorem 1.1]{AHR2} show that the energy of the system \eqref{AFF11} decreases to zero as $\ln^{-2}(t)$ as $t$  goes to $+\infty$ for regular initial data. In particular, it is proven that the resolvent estimate of the operator $\mathcal{A}_{0}$ satisfies for some constant $C>0$ large enough,
$$
\|(i\omega I-\mathcal{A}_{0})^{-1}\|_{\mathcal{L}(\mathcal{H}_{0})}\leq C\e^{C|\omega|},\quad\forall\,|\omega|\gg1.
$$
Hence, by combing this resolvent estimate with Corollary \ref{cor : logarithme stable} we obtain the following
\begin{prop}
Under the above assumption and for $\eta>0$ the operator $\mathcal{A}$ generates a  C$_0$ semigroup of contractions satisfying
$$
\|e^{t\mathcal{A}}X\|_{\mathcal{H}}\leq\frac{C}{\ln(1+t)}\|X\|_{\mathcal{D}(\mathcal{A})},\quad\forall\,X\in\mathcal{D}(\mathcal{A}),\;t\geq 0,
$$
for some constant $C>0$. This means that the energy of system \eqref{AFF11} is decreasing to zero as $t$ goes to $+\infty$ as $\ln^{-2}(t)$.
\end{prop}
\subsection{Pointwise fractional-damped string equation}
We consider the equation of the vibration of a string of length equal to $1$ with a pointwise fractional damping modeled by the following equation
\begin{equation*}
\left\{\begin{array}{ll}
\partial_{t}^{2}u(x,t)-u''(x,t)+\partial_{t}^{\alpha,\eta}u(\zeta,t)\delta_{\zeta}=0&(x,t)\in(0,1)\times\R_{+}
\\
u(0,t)=u(1,t)=0&t\in\R_{+}
\\
u(x,0)=u^{0}(x),\qquad \partial_{t}u(x,0)=u^{1}(x)&x\in(0,1),
\end{array}\right.
\end{equation*}
where the prime denotes the space derivative and $\delta_{\zeta}$ is the Dirac mass concentrated in the point $\zeta$ of $(0,1)$ (See \cite{hassine,tucsnak} for the classical derivative). Equivalently we have
\begin{equation}\label{AFF17}
\left\{\begin{array}{ll}
\ds\partial_{t}^{2}u(x,t)-u''(x,t)+\gamma\int_{\R}p(\xi)\varphi(t,\xi)\,\ud\xi\,\delta_{\zeta}=0&(x,t)\in(0,1)\times\R_{+}
\\
\partial_{t}\varphi(t,\xi)+(|\xi|^{2}+\eta)\varphi(t,\xi)=p(\xi)\partial_{t}u(\zeta,t)&(x,t,\xi)\in(0,1)\times\R_{+}\times\R
\\
u(0,t)=u(1,t)=0&t\in\R_{+}
\\
u(x,0)=u^{0}(x),\quad \partial_{t}u(x,0)=u^{1}(x),\quad \partial_{t}\varphi(0,\xi)=\varphi^{0}(\xi)&(x,\xi)\in(0,1)\times\R,
\end{array}\right.
\end{equation}
where we recall here that $U=\mathbb{C}$, $H=L^{2}(0,1)$, $H_{\frac{1}{2}}=H_{0}^{1}(0,1)$, $H_{-\frac{1}{2}}=H^{-1}(0,1)$, $Bz=z\delta_{\zeta}$ for all $z\in\mathbb{C}$ and $B^{*}u=u(\zeta)$ for all $u\in H_{0}^{1}(0,1)$.

We consider now the operator $\mathcal{A}:\mathcal{D}(\mathcal{A})\subset\mathcal{H}\longrightarrow\mathcal{H}$ defined by
$$
\mathcal{A}\left(\begin{array}{c}
u
\\
v
\\
\varphi
\end{array}\right)=\left(\begin{array}{c}
v
\\
\ds u''+\gamma\int_{\R}p(\xi)\varphi(\xi)\,\ud\xi\,\delta_{\zeta}
\\
-(|\xi|^2+\eta)\varphi(\xi)+p(\xi)v(\zeta)
\end{array}\right),
$$
in the Hilbert space $\mathcal{H}=H_{0}^{1}(0,1)\times L^{2}(0,1)\times L^{2}(\mathbb{C};\R)$ with domain
\begin{equation*}
\begin{split}
\ds\mathcal{D}(\mathcal{A})=\Big\{(u,v,\varphi)\in \mathcal{H}:v\in H_{0}^{1}(0,1),\; u''+\gamma\int_{\R}p(\xi)\varphi(\zeta,\xi)\,\ud\xi\,\delta_{\zeta}\in L^{2}(0,1),
\\
|\xi|\varphi\in L^{2}(\R;\mathbb{C}),\; -(|\xi|^2+\eta)\varphi(\xi)+p(\xi)v(\zeta)\in L^{2}(\R;\mathbb{C})\Big\}.
\end{split}
\end{equation*}
The energy of the solution of system \eqref{AFF17} is given by
$$
E(t)=\frac{1}{2}\left(\|\partial_{t}u(t)\|_{L^{2}(0,1)}^{2}+\|u'(t)\|_{L^{2}(0,1)}^{2}+\gamma\int_{\R}|\varphi(t,\xi)|^{2}\,\ud\xi\right).
$$
\begin{prop}
The semigroup generated by the operator $\mathcal{A}$ is strongly stable, i.e
$$
\lim_{t\to+\infty}\|e^{\mathcal{A}t}(u^{0},v^{0},\varphi^{0})\|_{\mathcal{H}}=0,\qquad\forall(u^{0},u^{1},\varphi^{0})\in \mathcal{H},
$$
if and only if $\zeta\notin\mathbb{Q}$.
\end{prop}
\begin{proof}
The prove is done in two stages:
\begin{itemize}
\item We consider the following problem
\begin{equation}\label{AFF18}
\left\{\begin{array}{ll}
\partial_{t}^{2}u-u''(x,t)=0&(x,t)\in(0,1)\times\R_{+}
\\
\partial_{t}u(\zeta,t)=0&
\\
u(0,t)=u(1,t)=0,&
\end{array}\right.
\end{equation}
Then the solution of \eqref{AFF18} is given by
\begin{equation}\label{AFF20}
\begin{split}
u(x,t)&=2\sum_{k=1}^{+\infty}\langle u^{0},\sin(k\pi\,.\,)\rangle_{L^{2}(0,1)}\cos(k\pi t)\sin(k\pi x)
\\
&+2\sum_{k=1}^{+\infty}\langle u^{1},\sin(k\pi\,.\,)\rangle_{L^{2}(0,1)}\frac{\sin(k\pi t)\sin(k\pi x)}{k\pi},\quad \forall\, x\in(0,1),\;\forall\,t\in\R_{+},
\end{split}
\end{equation}
where $u^{0}$ and $u^{1}$ are the initial data. In particular, we have 
\begin{equation*}
\begin{split}
\partial_{t}u(\zeta,t)&=-2\pi\sum_{k=1}^{+\infty}k\langle u^{0},\sin(k\pi\,.\,)\rangle_{L^{2}(0,1)}\sin(k\pi t)\sin(k\pi\zeta)
\\
&+2\sum_{k=1}^{+\infty}\langle u^{1},\sin(k\pi\,.\,)\rangle_{L^{2}(0,1)}\cos(k\pi t)\sin(k\pi\zeta)=0,\qquad\forall\,t\in\R_{+}.
\end{split}
\end{equation*}
The uniqueness of the Fourier series implies that $k\langle u^{0},\sin(k\pi\,.\,)\rangle_{L^{2}(0,1)}\sin(k\pi\zeta)=0$ and $\langle u^{1},\sin(k\pi\,.\,)\rangle_{L^{2}(0,1)}\sin(k\pi\zeta)=0$ for all $k\in\mathbb{N}^{*}$. Since $\zeta\notin\mathbb{Q}$ then $\sin(k\pi\zeta)\neq0$ for all $k\in\mathbb{N}^{*}$. Therefore, $\langle u^{0},\sin(k\pi\,.\,)\rangle_{L^{2}(0,1)}=0$ and $\langle u^{1},\sin(k\pi\,.\,)\rangle_{L^{2}(0,1)}=0$ for all $k\in\mathbb{N}^{*}$. Following to \eqref{AFF20} we obtain $u=0$. Thus, the first implication follows from Theorem \ref{SSFF20}.
\item We recall that the sequence of eigenfunctions of the Dirichlet Lapacian operator in $(0,1)$ are given by
$$
u_{k}(x)=\sin(k\pi x)\qquad\forall\,x\in(0,1)
$$
formed an orthonormal base of $L^{2}(0,1)$ with the corresponding eigenvalues $-\mu_{k}=-k^{2}$ for all $k\in\mathbb{Z}$. Since $\zeta\in\mathbb{Q}$ then $B^{*}u_{k}=\sin(k\pi\zeta)=0$ for some $k\in\N$. Following to the second item of Theorem \ref{LUSFF9} $ik$ is an eigenvalue of the operator $\mathcal{A}$. Therefore $\sigma(\mathcal{A})\cap i\R\neq \emptyset$. This prove the second implication.
\end{itemize}
This completes the proof.
\end{proof}

We consider now the following auxiliary problem
\begin{equation}\label{AFF23}
\left\{\begin{array}{ll}
\partial_{t}^{2}w(x,t)-w''(x,t)+\partial_{t}w(\zeta,t)\delta_{\zeta}=0&(x,t)\in(0,1)\times\R_{+}
\\
w(0,t)=w(1,t)=0&t\in\R_{+}
\\
w(x,0)=w^{0}(x),\qquad \partial_{t}w(x,0)=w^{1}(x)&x\in(0,1).
\end{array}\right.
\end{equation}
System \eqref{AFF23} is well posed in the Hilbert space $\mathcal{H}_{0}=H_{0}^{1}(0,1)\times L^{2}(0,1)$ and its solution is a semigroup generated by the operator $\mathcal{A}_{0}:\mathcal{D}(\mathcal{A})\subset\mathcal{H}_{0}\longrightarrow\mathcal{H}_{0}$ defined by
$$
\mathcal{A}_{0}\left(\begin{array}{c}
w
\\
v
\end{array}\right)=\left(\begin{array}{c}
v
\\
w''-v(\zeta)\delta_{\zeta}
\end{array}\right)
$$
with domain
$$
\mathcal{D}(\mathcal{A}_{0})=\{(w,v)\in\left[H_{0}^{1}(0,1)\right]^{2}:\;w\in H^{2}(0,\zeta)\cap H^{2}(\zeta,1),\;w'(\zeta^{+})-w'(\zeta^{-})=v(\zeta)\}.
$$
Considering the following subset
\begin{align*}
\mathcal{M}=\Big\{\zeta\in(0,1):\exists\,K_{1},\,K_{2}> 0,\,\left(\sin^{2}(\mu)+\sin^{2}(\zeta\mu).\sin^{2}((1-\zeta)\mu)\right)\e^{K_{1}\mu}\geq K_{2},\,\forall \mu\gg 1\Big\}.
\end{align*}
According to \cite[Theorem 1.1]{hassine}, if $\zeta\in\mathcal{M}$ the energy of the system \eqref{AFF23} is deceasing as $\ln^{-2}(t)$ as $t$ goes to $+\infty$ in particular it is proven that the resolvent estimate of the operator $\mathcal{A}_{0}$ satisfies for some constant $C>0$ large enough,
$$
\|(i\omega I-\mathcal{A}_{0})^{-1}\|_{\mathcal{L}(\mathcal{H}_{0})}\leq C\e^{C|\omega|},\quad\forall\,|\omega|\gg1.
$$
Therefore, by combining this estimation with Corollary \ref{cor : logarithme stable} we obtain the following
\begin{prop}
Under the above assumption on $\zeta$ and for $\eta>0$ the operator $\mathcal{A}$ generates a  C$_0$ semigroup of contractions satisfying
$$
\|e^{t\mathcal{A}}X\|_{\mathcal{H}}\leq\frac{C}{\ln(1+t)}\|X\|_{\mathcal{D}(\mathcal{A})},\quad\forall\,X\in\mathcal{D}(\mathcal{A}),\;t\geq 0,
$$
for some constant $C>0$. This means that the energy of system \eqref{AFF11} is decreasing to zero as $t$ goes to $+\infty$ as $\ln^{-2}(t)$.
\end{prop} 

\end{document}